\documentclass[11pt]{amsart}

\usepackage{url}
\usepackage{longtable, stackrel, multicol, fancybox, stackrel, tcolorbox, mathtools, adjustbox}
\usepackage{amsfonts, amsmath ,amssymb, mathrsfs, amsthm} 
\usepackage{graphicx, color, eurosym, eucal}
\usepackage{wasysym, calc, mathdots, physics}

\usepackage[normalem]{ulem}
\usepackage[marginpar=2cm]{geometry}
\usepackage{pgfplots}
\pgfplotsset{width=10cm,compat=1.9}

\usepackage{mathrsfs,multicol}
\usepackage[pagebackref, 
hypertexnames=false, colorlinks, citecolor=red, linkcolor=red]{hyperref}
\usetikzlibrary{decorations.pathmorphing}

	\normalbaselines						  %
\newcommand{\wt}{\widetilde}
\newcommand{\wh}{\widehat}

\def\RR{\mathbb{R}}

\def\CC{\mathbb{C}}
\def\NN{\mathbb{N}}

\newcommand{\la}{{\lambda}}
\newcommand{\f}{{\varphi}}

\newcommand{\R}{{\mathbb  R}}

\newcommand{\te}{{\theta}}

\newcommand{\N}{{\mathbb  N}}

\newcommand{\sgn}{\text{sgn}}

\makeatletter
\def\Ddots{\mathinner{\mkern1mu\raise\p@
\vbox{\kern7\p@\hbox{.}}\mkern2mu
\raise4\p@\hbox{.}\mkern2mu\raise7\p@\hbox{.}\mkern1mu}}
\makeatother

\newcommand{\cH}{\mathcal{H}}

\newcommand{\eps}{\varepsilon}

\newcommand{\bd}{\boldsymbol{d}}

\newcommand{\cL}{\mathcal{L}}

\DeclareMathOperator{\dom}{dom}

\DeclareMathOperator{\ran}{ran}
\DeclareMathOperator{\spa}{span}

\DeclareMathOperator{\supp}{supp}

\def\bm1{\mathbbm{1}}

\newcommand{\w}{\omega}

\newcommand{\ci}[1]{_{ {}_{\scriptstyle #1}}}
\newcommand{\ti}[1]{_{\scriptstyle \text{\rm #1}}}

\definecolor{darkblue}{rgb}{0.2,0.2,0.6}
\definecolor{darkgreen}{rgb}{0.2,0.6,0.2}

\newcommand{\Hm}[1]{\leavevmode{\marginpar{\tiny%
$\hbox to 0mm{\hspace*{-0.5mm}$\leftarrow$\hss}%
\vcenter{\vrule depth 0.1mm height 0.1mm width \the\marginparwidth}%
\hbox to
0mm{\hss$\rightarrow$\hspace*{-0.5mm}}$\\\relax\raggedright #1}}}

\catcode`@=11
\chardef\mathlig@atcode\count255

\def\actively#1#2{\begingroup\uccode`\~=`#2\relax\uppercase{\endgroup#1~}}
\def\mathlig@gobble{\afterassignment\mathlig@next@cmd\let\mathlig@next= }
\def\mathlig@delim{\mathlig@delim}
\def\mathlig@defcs#1{\expandafter\def\csname#1\endcsname}
\def\mathlig@let@cs#1#2{\expandafter\let\expandafter#1\csname#2\endcsname}
\def\mathlig@appendcs#1#2{\expandafter\edef\csname#1\endcsname{\csname#1\endcsname#2}}

\def\mathlig#1#2{\mathlig@checklig#1\mathlig@end\mathlig@defcs{mathlig@back@#1}{#2}\ignorespaces}

\def\mathlig@checklig#1#2\mathlig@end{%
 \expandafter\ifx\csname mathlig@forw@#1\endcsname\relax
 \expandafter\mathchardef\csname mathlig@back@#1\endcsname=\mathcode`#1%
 \mathcode`#1"8000\actively\def#1{\csname mathlig@look@#1\endcsname}%
 \mathlig@dolig#1\mathlig@delim
\fi
\mathlig@checksuffix#1#2\mathlig@end
}

\def\mathlig@checksuffix#1#2\mathlig@end{%
\ifx\mathlig@delim#2\mathlig@delim\relax\else\mathlig@checksuffix@{#1}#2\mathlig@end\fi
}
\def\mathlig@checksuffix@#1#2#3\mathlig@end{%
\expandafter\ifx\csname mathlig@forw@#1#2\endcsname\relax\mathlig@dosuffix{#1}{#2}\fi
\mathlig@checksuffix{#1#2}#3\mathlig@end
}

\def\mathlig@dosuffix#1#2{%
\mathlig@appendcs{mathlig@toks@#1}{#2}%
\mathlig@dolig{#1}{#2}\mathlig@delim
}

\def\mathlig@dolig#1#2\mathlig@delim{%
 \mathlig@defcs{mathlig@look@#1#2}{%
 \mathlig@let@cs\mathlig@next{mathlig@forw@#1#2}\futurelet\mathlig@next@tok\mathlig@next}%
 \mathlig@defcs{mathlig@forw@#1#2}{%
  \mathlig@let@cs\mathlig@next{mathlig@back@#1#2}%
  \mathlig@let@cs\checker{mathlig@chck@#1#2}%
  \mathlig@let@cs\mathligtoks{mathlig@toks@#1#2}%
  \expandafter\ifx\expandafter\mathlig@delim\mathligtoks\mathlig@delim\relax\else
  \expandafter\checker\mathligtoks\mathlig@delim\fi
  \mathlig@next
 }%
 \mathlig@defcs{mathlig@toks@#1#2}{}%
 \mathlig@defcs{mathlig@chck@#1#2}##1##2\mathlig@delim{%
  \ifx\mathlig@next@tok##1%
   \mathlig@let@cs\mathlig@next@cmd{mathlig@look@#1#2##1}\let\mathlig@next\mathlig@gobble
  \fi
  \ifx\mathlig@delim##2\mathlig@delim\relax\else
   \csname mathlig@chck@#1#2\endcsname##2\mathlig@delim
  \fi
 }%
 \ifx\mathlig@delim#2\mathlig@delim\else
  \mathlig@defcs{mathlig@back@#1#2}{\csname mathlig@back@#1\endcsname #2}%
 \fi
}%

\catcode`@\mathlig@atcode

\mathchardef\ordinarycolon\mathcode`\:
\def\vcentcolon{\mathrel{\mathop\ordinarycolon}}
\mathlig{:=}{\vcentcolon=}
\mathlig{::=}{\vcentcolon\vcentcolon=}

\numberwithin{equation}{section}

\theoremstyle{plain}
\newtheorem{theo}{Theorem}[section]
\newtheorem{cor}[theo]{Corollary}
\newtheorem{lem}[theo]{Lemma}
\newtheorem{prop}[theo]{Proposition}

\theoremstyle{definition}

\newtheorem*{theorem*}{Theorem}
\newtheorem*{idea*}{Idea}

\theoremstyle{remark}
\newtheorem{rem}[theo]{Remark}

\newtheorem*{ex*}{Example}
\newtheorem*{exs*}{Examples}
\newtheorem*{rems*}{Remarks}

\title[Spectrum of Dirac Operator
with singular interaction on broken line]{Spectral analysis of the Dirac operator with a singular interaction on a broken line}

\author{Dale~Frymark}
\address{Technische Universit\"at Graz, Institut f\"ur Angewandte Mathematik, Steyrergasse 30, 8010 Graz, Austria}

\email{frymark@math.tugraz.at}

\author{Markus Holzmann}
\address{Technische Universit\"at Graz, Institut f\"ur Angewandte Mathematik, Steyrergasse 30, 8010 Graz, Austria}
\email{holzmann@math.tugraz.at}

\author{Vladimir Lotoreichik}
\address{Department of Theoretical Physics, Nuclear Physics Institute, Czech Academy of Sciences, 25068 
\v{R}e\v{z}, Czech Republic      
}
\email{lotoreichik@ujf.cas.cz}


\begin{document}

\begin{abstract}
We consider the one-parametric family of  self-adjoint realizations of the two-dimen\-sional massive Dirac operator with a Lorentz scalar $\delta$-shell interaction of strength $\tau\in\RR\setminus\{-2,0,2\}$ supported on a broken line of opening angle $2\w$ with $\w\in(0,\frac{\pi}{2})$. The essential spectrum of any such self-adjoint realization is symmetric with respect to the origin with a gap around zero whose size depends on the mass and, for $\tau < 0$, also on the strength of the interaction, but does not depend on $\w$. As the main result, we prove that
for any $N\in\NN$ and strength $\tau\in(-\infty,0)\setminus\{-2\}$ the discrete spectrum of any such self-adjoint realization has at least $N$ discrete eigenvalues, with multiplicities taken into account, in the gap of the essential spectrum provided that $\w$ is sufficiently small. Moreover, we obtain an explicit estimate on $\w$ sufficient for this property to hold. For $\tau\in(0,\infty)\setminus\{2\}$, the discrete spectrum consists of at most one simple eigenvalue. 
\end{abstract}

\maketitle

\setcounter{tocdepth}{1}
\tableofcontents

\section{Introduction}\label{s-intro}

Dirac operators with singular interactions supported on hypersurfaces
in the Euclidean space serve
as an idealized model for more physically realistic Dirac operators with regular potentials localized in the vicinity of a hypersurface. The Dirac operator with a singular interaction
supported on a sphere in $\RR^3$ was first analysed in~\cite{DES}, where  separation of variables was extensively used. The renewal of interest to this topic is connected with the papers~\cite{AMS,AMS3,AMS2}, where singular interactions supported on general closed surfaces in $\RR^3$ were analysed, see also \cite{BEHL1,BEHL2} for further considerations in dimension three, \cite{BHOBP, CLMT} for the study of the two-dimensional case, and \cite{BHSS} for a recent unified treatment for space dimensions two and three. 
  
In this paper we are interested in spectral properties of Dirac operators with singular potentials supported on a broken line, for which there
are several motivations. On the one hand,
the analysis of two-dimensional Dirac operators with singular interactions supported on curves with corner points revealed the importance of the model Dirac operators with singular interactions supported on unbounded broken lines~\cite{TOB, PVDB}. In particular, self-adjointness properties of such Dirac operators with Lorentz scalar $\delta$-shell interactions supported on star graphs were studied in \cite{FL}.

On the other hand, the broken line is a typical example where the existence of geometrically induced eigenvalues is observed. 
The relation of the spectral properties of a differential operator and the geometry of the interaction support or the domain in which the operator acts is one of the most classical questions in spectral theory. In particular, one is often interested in situations when a modification of the geometry generates bound states.  There is a huge number of results available for Laplacians with different boundary conditions, see the monograph \cite{EK15} and the references therein or~\cite{DE95, DEK01} for the case of Dirichlet boundary conditions,~\cite{EM14, P16} for related results for Laplacians with Robin boundary conditions, and~\cite{BFKR22, ELP18} for considerations on the magnetic Laplacian with Neumann boundary conditions.   
One interesting example of the above type is the non-relativistic counterpart of the Dirac operator with a singular potential. Consider a Schr\"odinger operator with a $\delta$-potential of strength $\alpha \in \mathbb{R}$ supported on a broken line $\Gamma$  of opening angle $2 \omega$ with $\omega \in (0,\frac{\pi}{2}]$, i.e. the operator that is formally given by
\begin{equation*}
  T_\alpha = -\Delta + \alpha \delta_\Gamma.
\end{equation*}
With the help of the Birman-Schwinger principle it was established in \cite{EI01} that $\sigma(T_\alpha) = \sigma_{\textup{ess}}(T_\alpha) = [0, \infty)$ for $\alpha \geq 0$, while for attractive interactions $\alpha < 0$ one has $\sigma_{\textup{ess}}(T_\alpha) = [-\frac{\alpha^2}{4}, \infty)$ and, if $\Gamma$ is not a straight line, then the discrete spectrum of $T_\alpha$ is non-empty. In \cite{EN03} it was shown  with the help of the min-max principle and suitable test functions  that for any $N \in \mathbb{N}$ the operator $T_\alpha$ has at least $N$ discrete eigenvalues (taking multiplicities into account), if the opening angle of $\Gamma$ is sufficiently small. Finally, it was proved in \cite{P17} with a variational argument  that $T_\alpha$ always has, for $\alpha < 0$, a non-empty discrete spectrum, when the interaction is supported on  a star graph which is not a straight line, see also \cite{PV22} for related considerations for Schr\"odinger operators with $\delta'$-interactions.

The main goal of this paper is to study the spectral properties and to establish the existence of geometrically induced bound states for the two-dimensional Dirac operator with a Lorentz scalar $\delta$-shell interaction supported on the broken line $\Gamma$ that is formally given by
\begin{equation} \label{def_A_tau_formal}
  \mathcal{A}_\tau = -i \sigma_1 \partial_1 - i \sigma_2 \partial_2 + m \sigma_3 + \tau \sigma_3 \delta_\Gamma.
\end{equation}
Here, $\tau \in \mathbb{R} \setminus \{ -2, 0, 2 \}$, $m > 0$, and $\sigma_1, \sigma_2$, and $\sigma_3$ are the Pauli spin matrices defined in~\eqref{def_Pauli} below. From a physical point of view, $\mathcal{A}_\tau$ describes the propagation of a relativistic quantum particle with mass $m$ under the influence of the Lorentz scalar potential $\tau \sigma_3 \delta_\Gamma$.  We remark that $\tau = \pm 2$ is excluded, as in this case the operator $\mathcal{A}_\tau$ decouples into two Dirac operators acting in the domains with boundary $\Gamma$ and infinite mass boundary conditions on
$\Gamma$, cf. \cite{BHOBP, BHSS}. In particular, in this \textit{confinement} case the spectrum is computed in \cite[Proposition~1.12]{TOB}, by which the essential spectrum of any self-adjoint realization of $\mathcal{A}_2$ is $(-\infty, -m] \cup [m, \infty)$ and there can only be one eigenvalue having multiplicity one, while the spectrum of any self-adjoint realization of $\mathcal{A}_{-2}$ is always the full real line; note that $\tau=-2$ in the present manuscript corresponds to $m<0$ in \cite{TOB}, which can be achieved via a unitary transformation.


Let us describe our results.
The broken line decomposes $\RR^2$ into two regions. Let
\begin{align} \label{def_Omega_+}
    \Omega_+:=\left\{(r\cos\te,r\sin\te)\in\RR^2\colon r>0,~\te\in (-\omega, \omega) \right\}\subset\RR^2,
\end{align}
be the wedge of opening angle $2 \omega$, and $\Omega_- := \mathbb{R}^2 \setminus \overline{\Omega_+}$ be the open complement of $\Omega_+$. Distinguish $\Gamma := \partial \Omega_+$ as the joint boundary of $\Omega_\pm$, and let $\nu = (\nu_1, \nu_2)$ be the unit normal vector field at $\Gamma$ which is pointing outwards of $\Omega_+$. Note that setups given with angles $\omega\in (\frac{\pi}{2},\pi)$ can have $\Omega_+$ and $\Omega_-$ interchanged and then rotated to yield a unitarily equivalent setup with $\omega\in (0,\frac{\pi}{2}]$, as above, with the same spectrum. Hence, we will always assume $\omega\in (0,\frac{\pi}{2}]$ in this paper.

In the following, functions $u\colon \mathbb{R}^2 \rightarrow \mathbb{C}^2$ will often be restricted to the two regions of $\RR^2$ and this is denoted by $u_\pm := u|_{\Omega_\pm}$. In order to define the operator in~\eqref{def_A_tau_formal} rigorously, we introduce, for $\tau \in \mathbb{R} \setminus \{ -2, 0, 2 \}$ and $m > 0$, the operator
\begin{equation} \label{def_S_tau}
  \begin{split}
    S_\tau u &= (-i \sigma_1 \partial_1 - i \sigma_2 \partial_2 + m \sigma_3) u_+ \oplus (-i \sigma_1 \partial_1 - i \sigma_2 \partial_2 + m \sigma_3) u_-, \\
    \dom S_\tau &= \left\{ u = u_+ \oplus u_- \in H^1(\Omega_+; \mathbb{C}^2) \oplus H^1(\Omega_-; \mathbb{C}^2): u_- |_\Gamma = M u_+ |_\Gamma \right\},
  \end{split}
\end{equation}
where the matrix-valued function $M$ is given by
\begin{equation} \label{def_M}
  M = \frac{4+\tau^2}{4-\tau^2} \sigma_0 + \frac{4 \tau}{4 - \tau^2} i \sigma_3 (\sigma_1 \nu_1 + \sigma_2 \nu_2)
\end{equation}
and $\sigma_0$ is the $2 \times 2$ identity matrix.

If $\Omega_+$ is a smooth bounded domain, then the self-adjointness and some spectral properties of $S_\tau$--including finiteness of the discrete spectrum--were studied in \cite{BHOBP}, see also \cite{AMS3, BEHL2, BHSS} for related results in dimension three. In \cite{HOBP}, the three-dimensional counterpart of $S_\tau$ with $\Omega_+$ being a bounded $C^4$-domain was studied with the help of a quadratic form associated with $S_\tau^2$ and for $\tau<0$ the asymptotic expansion of the discrete eigenvalues was shown for $m \rightarrow \infty$, see also \cite{ALTR17, LO18} for related considerations on Dirac operators on domains with boundary conditions. Some of the ideas from \cite{HOBP} are used in the analysis in the present paper as well. The case that $\Gamma \subset \mathbb{R}^2$ is the straight line was considered in \cite{BHT23}. Eventually, we mention that in the papers \cite{ABLO23, BFSVDB2, LO18} results on the relation of the geometry of a domain $\Omega \subset \mathbb{R}^2$ and the spectrum of the Dirac operator acting in $\Omega$ with infinite mass boundary conditions are provided.
In the case of Dirac operators with zigzag boundary conditions there is a clear picture of the geometrically induced spectrum \cite{EH22, S95}, as the spectrum of this operator can be expressed via the spectrum of the Dirichlet Laplacian acting in the same domain. However, to the best of our knowledge, there is no result in the literature where the existence of geometrically induced bound states of Dirac operators with singular potentials is shown, which is the main motivation in the present paper.

Coming back to the case that $\Gamma$ is the broken line, it is known that the operator $S_\tau$ is closed and symmetric, cf.~\cite{FL}. Note that for $\omega = \frac{\pi}{2}$ the curve $\Gamma$ is a straight line. In this case it is known that $S_\tau$ is self-adjoint and that the spectrum of $S_\tau$ is purely absolutely continuous, cf. \cite{BHT23}.
If $\omega \in(0, \frac{\pi}{2})$, then the deficiency indices of $S_\tau$ are $\dim \ker(S_\tau^* \mp i)=1$. Hence, $S_\tau$ admits a one-parametric family, labeled via $z$, of self-adjoint extensions $A_{\tau, z}$, cf. \cite{FL} and Section~\ref{s-wedge} for details. The spectral properties of these extensions are contained in the following, central result of the article.

\begin{theo} \label{theorem_intro}
  Let $\tau \in \mathbb{R} \setminus \{-2, 0, 2 \}$, $\omega \in (0, \frac{\pi}{2})$,  $m>0$, and $A_{\tau,z}$ be any self-adjoint extension of $S_\tau$ in~\eqref{def_S_tau}. Then, the following is true:
  \begin{itemize}
    \item[(i)] If $\tau > 0$, then $\sigma_\textup{ess}(A_{\tau,z}) = (-\infty, -m] \cup [m, \infty)$ and $A_{\tau,z}$ has at most one discrete eigenvalue of multiplicity one. 
    \item[(ii)] If $\tau < 0$, then $\sigma_\textup{ess}(A_{\tau,z}) = \big(-\infty, -m \big|\frac{\tau^2 - 4}{\tau^2+4}\big| \big] \cup \big[m \big|\frac{\tau^2 - 4}{\tau^2+4}\big|, \infty\big)$ and for any $N \in \mathbb{N}$ there exists an angle $\omega_\star \in (0, \frac{\pi}{2})$ depending on $N, \tau$ such that for all $\omega \in (0, \omega_\star]$ the operator $A_{\tau,z}$ has at least $N$ discrete eigenvalues with multiplicities taken into account.
  \end{itemize}
\end{theo}

The result relies on an analysis of the quadratic form of $S_\tau^*S_\tau$, which is a restriction of the quadratic form of $(A_{\tau,z})^2$, see Section \ref{ss-quadratic}. Repercussions for the essential spectrum of $A_{\tau,z}$ are then obtained using perturbation theory and an auxiliary result on the symmetry of the essential spectrum with respect to the origin in Section \ref{section_symmetry_spectrum}. The statements on the essential spectrum in the theorem are then obtained in Section \ref{section_sigma_ess} using Neumann bracketing and the method of singular sequences.

The claim on the discrete spectrum for $\tau > 0$ in Theorem~\ref{theorem_intro}\,(i) relies on the abstract perturbation theory of self-adjoint operators. The result on the discrete spectrum for $\tau < 0$ in Theorem~\ref{theorem_intro}\,(ii) is shown in Section~\ref{section_discrete} with the min-max principle applied to the quadratic form of $S_\tau^*S_\tau$. An appropriate family of test functions is constructed via a modification of the approach used in~\cite{EN03} to deal with the Schr\"odinger operator with a $\delta$-interaction supported on a broken line. Moreover, we provide a formula for a suitable value of the critical angle $\omega_\star$. 
It remains an open problem whether the discrete spectrum of $A_{\tau,z}$ is always non-empty provided that $\omega\ne \frac{\pi}{2}$ and $\tau \in (-\infty,0)\setminus\{-2\}$.



\subsection*{Notations}

For a Hilbert space $\mathcal{H}$ we denote the associated inner product (which is linear in the first entry) by $\langle \cdot, \cdot\rangle_{\mathcal{H}}$  and the induced norm by $\| \cdot \|_{\mathcal{H}}$.
If $A$ is a closed linear operator in a Hilbert space $\mathcal{H}$, then we denote its domain, range, kernel, spectrum, and resolvent set by $\dom A$, $\ran A$, $\ker A$, $\sigma(A)$, and $\rho(A)$, respectively. If $A$ is self-adjoint, then its discrete and essential spectrum are $\sigma_\textup{disc}(A)$ and $\sigma_\textup{ess}(A)$, respectively.

Let
\begin{equation} \label{def_Pauli} 
\sigma_1 = \begin{pmatrix}
0 & 1\\                                              
1 & 0                                           
\end{pmatrix}, \quad \sigma_2 =  \begin{pmatrix}
0 & -i\\                                              
i & 0                                  
\end{pmatrix}, \quad \sigma_3 = 
\begin{pmatrix} 
1 & 0\\                                              
0 & -1 \\                                            
\end{pmatrix},
\end{equation}
be the Pauli spin matrices and denote the $2 \times 2$ identity matrix by $\sigma_0$. They satisfy the anti-commutation relations 
\begin{equation} \label{Pauli_anti_commutation}
  \sigma_j \sigma_k + \sigma_k \sigma_j = 2 \delta_{jk} \sigma_0, \quad j, k \in \{ 1,2,3 \}.
\end{equation}
The notation
\begin{equation*}
  \sigma \cdot \grad := \sigma_1 \partial_1 + \sigma_2 \partial_2 \quad \text{and} \quad \sigma \cdot x = \sigma_1 x_1 + \sigma_2 x_2,
\end{equation*}
will also be used for $x = (x_1,x_2) \in \mathbb{C}^2$. 

If $\Omega \subset \mathbb{R}^2$ is a sufficiently smooth domain, then $L^2(\Omega; \mathbb{C}^n)$ and $L^2(\partial \Omega; \mathbb{C}^n)$, $n\in\mathbb{N}$, denote the sets of all square-integrable $\mathbb{C}^n$-valued functions on $\Omega$ and $\partial \Omega$. To simplify notation, we write $\| \cdot \|_\Omega := \| \cdot \|_{L^2(\Omega; \mathbb{C}^n)}$ and $\| \cdot \|_{\partial \Omega} := \| \cdot \|_{L^2(\partial \Omega; \mathbb{C}^n)}$. Moreover, we use for $s>0$ the symbols $H^s(\Omega; \mathbb{C}^n)$ and $H^s(\partial \Omega; \mathbb{C}^n)$ for the $L^2$-based Sobolev spaces on $\Omega$ and $\partial \Omega$, respectively. For $s \in [0,1]$, we will need the space $H^{-s}(\partial \Omega; \mathbb{C}^n)$, which is defined as the dual space of $H^{s}(\partial \Omega; \mathbb{C}^n)$, and we denote the associated sesquilinear duality product by $\langle \cdot, \cdot \rangle_{H^{s}(\partial \Omega; \mathbb{C}^n) \times H^{-s}(\partial \Omega; \mathbb{C}^n)}$. For a function $u \in H^1(\Omega; \mathbb{C}^n)$ we denote its Dirichlet trace by $u|_{\partial \Omega} \in H^{1/2}(\partial \Omega; \mathbb{C}^n)$.

Recall that the wedge $\Omega_+$ is defined in~\eqref{def_Omega_+}, that $\Omega_- = \mathbb{R}^2 \setminus \overline{\Omega_+}$, and $\nu$ is the outer unit normal vector field at $\Gamma := \partial \Omega_+$ that is pointing outwards from $\Omega_+$, cf. Figure~\ref{f-geometry}.
\begin{figure}[h]
\centering
\includegraphics[scale=1]{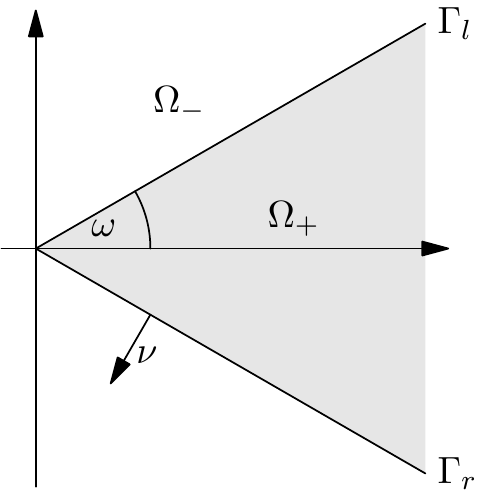}
\caption{Geometric setting. The domain $\Omega_+$ is depicted in gray.}
\label{f-geometry}
\end{figure}
We are also going to make use of the notations
\begin{equation} \label{def_Gamma_l_r}
  \begin{split}
    \Gamma_l := \big\{ (r \cos \omega, r \sin \omega): r > 0 \big\} \quad \text{and} \quad
    \Gamma_r := \big\{ (r \cos \omega, -r \sin \omega): r > 0 \big\},
  \end{split}
\end{equation}
so that $\Gamma = \overline{\Gamma_l \cup \Gamma_r}$.
A function (spinor) $u\in L^2(\RR^2;\CC^2)$ can be decomposed into two parts corresponding to the domains $\Omega_\pm$,
\begin{align*}
    u=u_+\oplus u_-\in L^2(\Omega_+;\CC^2)\oplus L^2(\Omega_-;\CC^2),~~\text{ where }~~ u_\pm :=u|_{\Omega_\pm}.
\end{align*}

\section{Singular interaction supported on a broken line}\label{s-wedge}

Self-adjoint extensions of the operator $S_\tau$ in equation \eqref{def_S_tau} were studied in \cite{FL, PVDB} along with the natural analog of $S_\tau$ with the Lorentz scalar $\delta$-shell interaction supported on the star graph~\cite{FL} and the boundary of a curvilinear polygon~\cite{PVDB}. These self-adjoint extensions are the starting point for the spectral analysis in Sections \ref{section_symmetry_spectrum}-\ref{section_discrete}.

Recall that $u \in H^1(\mathbb{R}^2 \setminus \Gamma; \mathbb{C}^2)$ belongs to $\dom S_\tau$ if and only if $u_-|_\Gamma = M u_+|_\Gamma$, where the matrix $M$ is defined by~\eqref{def_M}. We remark that for $\tau \neq \pm 2$ the matrix $M$ is invertible and the transmission condition is equivalent to
\begin{equation} \label{jump_condition1}
  u_+|_\Gamma = M^{-1} u_-|_\Gamma \quad \text{with} \quad M^{-1} = \frac{4+\tau^2}{4-\tau^2} \sigma_0 - \frac{4 \tau}{4 - \tau^2} i \sigma_3 (\sigma \cdot \nu)
\end{equation}
and also to
\begin{align}\label{e-transmission}
-i(\sigma \cdot \nu)\left(u_+|_{\Gamma}-u_-|_{\Gamma}\right)=\frac{\tau\sigma_3}{2}\left(u_+|_{\Gamma}+u_-|_{\Gamma}\right).
\end{align}

The deficiency indices of $S_\tau$ were identified in \cite[Corollary 2.16\,(i)]{FL} by switching to polar coordinates and analyzing the orthogonal decomposition of $S_\tau$. This orthogonal decomposition is only possible for $m=0$, which we will assume in the following. However, since $m \sigma_3$ is a bounded perturbation, this does not affect the results about self-adjointness. Moreover, since the straight line was already fully studied in \cite{BHT23}, we will assume that $\omega \in (0,\frac{\pi}{2})$ in the following.

The angular component of the aforementioned orthogonal decomposition of $S_\tau$ is commonly referred to as the spin-orbit operator and is defined by the linear map
\begin{equation}\label{e-angular}
\begin{aligned}
J_\tau\f&:=\left(-i\sigma_3\f_+'+\frac{\f_+}{2}
\right)\oplus\left(-i\sigma_3\f_-'+\frac{\f_-}{2}
\right), \\
\dom J_\tau&:=\Big\{\f=\f_+\oplus\f_-\in H^1((-\omega, \omega); \CC^2)\oplus H^1((\omega, 2 \pi - \omega); \CC^2)\colon \\
&\hspace{12em}
\f_-(\omega)=M_l
\f_+(\omega),~ 
\f_-(2\pi-\omega)=M_r
\f_+(-\omega)\Big\},
\end{aligned}
\end{equation}
where the matrices $M_l$ and $M_r$ are the restrictions of $M$ in~\eqref{def_M} onto $\Gamma_l$ and $\Gamma_r$, respectively, and  given by
\begin{equation} \label{def_M_l_r}
  M_l = \frac{1}{1-\frac{1}{4}\tau^2}
\begin{pmatrix}
1+\frac{1}{4}\tau^2 & e^{-i\w}\tau \\
e^{i\w}\tau & 1+\frac{1}{4}\tau^2
\end{pmatrix}
\quad \text{and} \quad
M_r = \frac{1}{1-\frac{1}{4}\tau^2}
\begin{pmatrix}
1+\frac{1}{4}\tau^2 & -e^{i\w}\tau \\
-e^{-i\w}\tau & 1+\frac{1}{4}\tau^2
\end{pmatrix}.
\end{equation}
It was established in~\cite[Proposition 2.3]{FL} that the operator $J_\tau$ is self-adjoint in the Hilbert space $L^2((-\w,\w);\CC^2)\oplus L^2((\w,2\pi-\w);\CC^2)$ with purely discrete spectrum.
Spectral properties of $J_\tau$ were determined in \cite[Propositions 2.8, 2.11, 2.12]{FL} and will play a role later.
In the following, we write $n_\lambda := \dim \ker(J_\tau-\lambda)$, $\lambda \in\RR$.
\begin{lem}\label{l-eigenvalues}
Let $\w\in(0,\frac{\pi}{2})$ and $\tau\in\RR\setminus\{-2,0,2\}$. Then, $\sigma(J_\tau)$ is symmetric about the origin,
$n_\la = n_{-\la}\le2$ for any $\la\in\sigma(J_\tau)$, $0,\frac12\notin\sigma(J_\tau)$, and $J_\tau$ has one simple eigenvalue within the interval $(0,\frac{1}{2})$.
\end{lem}
One can then show that $S_\tau$ (for $m=0$) is given by the orthogonal sum of half-line Dirac operators with off-diagonal Coulomb potentials. These half-line Dirac operators are closed, symmetric and densely defined operators in the Hilbert space $L^2(\RR_+;\CC^2)$. They are denoted, for $\la> 0$, $\la\ne \frac12$, by
\begin{equation}\label{eq:1dDirac}
\bd_\la\psi :=
\begin{pmatrix}
0 & -\frac{d}{dr}-\frac{\la}{r} \\
\frac{d}{dr}-\frac{\la}{r} &  0
\end{pmatrix}\psi,\qquad
\dom\bd_\la:=H_0^1(\RR_+;\CC^2).\\
\end{equation}  
For $m=0$ the operator $S_\tau$ can then be decomposed by \cite[Theorem 2.14]{FL} as
\[
	S_\tau \simeq\bigoplus_{\la\in\sigma(J_\tau)\cap(0,\infty)} \bigoplus_{j=1}^{n_\lambda}\,\bd_\la.
\]
In particular, each $\bd_\la$ is self-adjoint when $\la>1/2$. When $\la\in(0,1/2)$, the deficiency indices of $\bd_\la$ are $(1,1)$. In particular, this shows that $S_\tau$ has also deficiency indices $(1,1)$. The corresponding deficiency elements are expressed in \cite[Lemma 4.1]{FL} in terms of modified Bessel functions of the second kind and order $\nu\in\RR$, denoted $K_\nu(\cdot)$. These modified Bessel functions play a key role in determining self-adjoint extensions of $S_\tau$ in~\cite[Theorem 4.2]{FL}, a simplified variant of which is formulated below. 
For its formulation, let $\varphi_\la$ be the eigenfunction of $J_\tau$ corresponding to the unique simple eigenvalue $\la$ lying in the interval $(0,1/2)$; here
we extend $\varphi_\la$ to the interval $[0,2\pi)$ by periodicity, i.e.~$\varphi_\la(\te+2\pi) = \varphi_\la(\te)$. We then define the functions
\begin{equation} \label{def_v_pm}
  v_\pm(r \cos \te, r \sin \te) = K_{\la-1/2}(r)\f_\la(\te) \pm K_{\la+1/2}(r)(\sigma_1 \cos\te + \sigma_2 \sin\te) \f_\la(\te), 
\end{equation}
for $r > 0$, $\te \in [0, 2 \pi)$.

\begin{theo}\label{t-nbreakdown}
	Let $\omega \in (0, \frac{\pi}{2})$, $\tau\in\mathbb{R} \setminus \{-2,0,2\}$, 
	and $v_\pm$ be given by~\eqref{def_v_pm}. All self-adjoint extensions of $S_\tau$ are in one-to-one correspondence with complex numbers $z$ of modulus one. Moreover, the self-adjoint extension $A_{\tau,z}$ of $S_\tau$ corresponding to $z$ is characterized by 
	\[
	\begin{aligned}
	    \dom A_{\tau,z} &= 
	    \bigg\{u=u_0+cv_+ + cz v_-\in L^2(\RR^2;\CC^2)\colon u_0\in \dom S_\tau, ~c \in \mathbb{C}  \bigg\},
	      \\
	      A_{\tau,z}u &=A_{\tau,z} (u_0+ cv_+ + cz v_-) = S_\tau u_0 + (i \sigma_0 + m \sigma_3) cv_+ + (-i \sigma_0 + m \sigma_3) czv_-.
	\end{aligned}
	\]
\end{theo}
\begin{rem}
The action of any self-adjoint extension $A_{\tau,z}$ of $S_\tau$ can be easily determined via the definition of distributional derivatives as
  \begin{equation*}
    A_{\tau,z} u = (-i \sigma \cdot \grad + m \sigma_3) u_+ \oplus (-i \sigma \cdot \grad + m \sigma_3) u_-,\quad u = u_+ \oplus u_- \in \dom A_{\tau,z},
  \end{equation*}
  where  $\sigma \cdot \grad u_\pm$ has to be understood in the distributional sense.
\end{rem}

The decomposition of general functions in the domain of $A_{\tau,z}$ given in Theorem \ref{t-nbreakdown} provides a tool to classify the regularity of functions in the domain. It is even possible to single out a unique distinguished self-adjoint extension whose operator domain is the most regular in the scale of Sobolev spaces. A similar statement is made for Dirac operators with infinite mass boundary conditions in sectors in \cite[Proposition 1.8]{TOB}. For a proof of the following result, we refer to \cite[Theorem~4.4]{FL}. We remark that this distinguished self-adjoint extension will also play a role in our spectral analysis in Section~\ref{section_symmetry_spectrum}.

\begin{theo}\label{t-distinguished}
	Let $\omega \in (0, \frac{\pi}{2})$ and $\tau\in\mathbb{R} \setminus \{-2,0,2\}$. The so called `distinguished' self-adjoint extension of $S_\tau$, the only extension such that 
	\begin{align*}
	    \dom A_{\tau,z}\subset H^{1/2}(\Omega_+; \CC^2)\oplus H^{1/2}(\Omega_-; \CC^2),
	\end{align*}
	is the extension corresponding to $z=1$. In fact, the slightly stronger statement 
	\begin{align*}
	\dom A_{\tau,1}\subset H^s(\Omega_+; \CC^2)\oplus H^s(\Omega_-; \CC^2),
	\end{align*}
	holds for all $s<1/2+\min\{\sigma(J_\tau)\cap(0,\infty)\}$.
\end{theo}

\section{Quadratic Form}\label{ss-quadratic}

Recall that for  $\tau \in \mathbb{R} \setminus \{ -2, 0, 2 \}$ the operator $S_\tau$  is defined by equation \eqref{def_S_tau}. The main goal in this section is to find a convenient representation of $\| S_\tau u \|_{\mathbb{R}^2}^2$ for $u \in \dom S_\tau$, see Theorem~\ref{t-quadraticform}, which is the quadratic form associated with the self-adjoint operator $S^*_\tau S_\tau$. This operator will play a central role in the spectral analysis of the self-adjoint extensions $A_{\tau,z}$ of $S_\tau$, as, e.g., the difference of the resolvents of $S^*_\tau S_\tau$ and $(A_{\tau,z})^2$ will be compared in Section~\ref{section_sigma_ess} to determine the essential spectrum of $A_{\tau,z}$. Moreover, this quadratic form will play a crucial role in the proof of the statements about the discrete spectrum of $A_{\tau,z}$ in Theorem~\ref{theorem_intro}. 

In the following we are going to make use of the tangential derivative $\partial_t$ on the boundary of a Lipschitz hypograph $\Omega \subset \mathbb{R}^2$. For its definition, let $\gamma\colon\mathbb{R} \rightarrow \mathbb{R}^2$ be an arc-length parametrization of $\partial \Omega$ such that the unit normal vector field $\nu$, which is pointing outwards of $\Omega$, is given by $\nu = (-\dot{\gamma}_2, \dot{\gamma}_1)$ almost everywhere. 
Then, for $\varphi \in H^1(\partial \Omega)$ we define 
\begin{equation} \label{def_partial_t}
  \partial_t \varphi(x) := ( \varphi \circ \gamma)'(x), \qquad x \in \mathbb{R}.
\end{equation}
Clearly, for any $s \in [0, 1]$ the tangential derivative $\partial_t$ gives rise to a bounded operator
\begin{equation*}
  \partial_t: H^s(\partial \Omega) \rightarrow H^{s-1}(\partial \Omega).
\end{equation*}
We will use that the tangential derivative of the trace of a function $u \in H^2(\Omega)$ is given by
\begin{equation*}
  \partial_t (u|_{\partial \Omega}) = \dot{\gamma} \cdot (\grad u)|_{\partial \Omega}  = \nu_2 (\partial_1 u)|_{\partial \Gamma} - \nu_1 (\partial_2 u)|_{\partial \Gamma}.
\end{equation*}
Now, we are prepared to state some of the finer details for the proof of Theorem~\ref{t-quadraticform}.

\begin{lem}\label{l-extension}
Let $\Omega\subset\RR^2$ be a Lipschitz hypograph and let $\nu=(\nu_1, \nu_2)$ be the unit normal vector field at $\partial \Omega$ which points outwards of $\Omega$. Then for all $f,g\in H^1(\Omega;\CC^2)$ the following holds:
\begin{align}\label{e-extension}
\big\langle(\sigma\cdot\grad)f,(\sigma\cdot\grad)g\big\rangle_{\Omega}=\big\langle\grad f,\grad g\big\rangle_{\Omega}-\big\langle f|_{\partial \Omega}, i\sigma_3 \partial_t (g|_{\partial \Omega})\big\rangle_{H^{1/2}(\partial \Omega; \mathbb{C}^2)\times H^{-1/2}(\partial \Omega; \mathbb{C}^2)}.
\end{align}
\end{lem}

\begin{proof}
Let $f\in H^1(\Omega;\CC^2)$ and $g\in H^2(\Omega;\CC^2)$. Then, the identity $(\sigma \cdot \grad)^2 = \Delta$ and an application of Gauss' theorem yield
\begin{align*}
    \big\langle(\sigma\cdot\grad)f,(\sigma\cdot\grad)g\big\rangle_{\Omega}&=\big\langle(\sigma\cdot\nu)f|_{\partial \Omega},((\sigma\cdot\grad)g)|_{\partial \Omega}\big\rangle_{\partial\Omega}-\big\langle f,(\sigma\cdot\grad)^2g\big\rangle_{\Omega} \nonumber \\
    &=\big\langle f|_{\partial \Omega},(\sigma\cdot\nu)((\sigma\cdot\grad)g)|_{\partial \Omega}-\nu\cdot\grad g|_{\partial \Omega}\big\rangle_{\partial\Omega}+\big\langle \grad f,\grad g\big\rangle_{\Omega} \nonumber \\
    &=\big\langle f|_{\partial \Omega},\sigma_1\nu_1\sigma_2(\partial_2 g)|_{\partial \Omega}+\sigma_2\nu_2\sigma_1(\partial_1g)|_{\partial \Omega}\big\rangle_{\partial\Omega}+\big\langle \grad f,\grad g\big\rangle_{\Omega} \nonumber \\
    &=\big\langle f|_{\partial \Omega},\sigma_1 \sigma_2\left(\nu_1(\partial_2 g)|_{\partial \Omega}-\nu_2(\partial_1 g)|_{\partial \Omega}\right)\big\rangle_{\partial\Omega}+\big\langle \grad f,\grad g\big\rangle_{\Omega} \\
    &=-\big\langle f|_{\partial \Omega},\sigma_1 \sigma_2 \partial_t (g|_{\partial \Omega})\big\rangle_{\partial\Omega}+\big\langle \grad f,\grad g\big\rangle_{\Omega}. 
\end{align*}
This expression agrees with equation \eqref{e-extension}, as $\sigma_1\sigma_2=i\sigma_3$. Since $H^2(\Omega; \mathbb{C}^2)$ is dense in $H^1(\Omega; \mathbb{C}^2)$, a continuity argument yields the claim for $g \in H^1(\Omega; \mathbb{C}^2)$.
\end{proof}

It is now possible to show a useful representation for the quadratic form associated with the self-adjoint operator $S_\tau^* S_\tau$, i.e. of the expression 
\begin{align*}
\|S_\tau u\|^2_{\RR^2}&=\|(-i\sigma\cdot\grad)u+m\sigma_3 u\|^2_{\RR^2\setminus\Gamma}, \qquad u \in \dom S_\tau.
\end{align*}
The following result can be seen as a two-dimensional variant of \cite[Proposition~3.1]{HOBP} with the interaction support being the broken line; cf. also  \cite[Theorem~1.5]{ALTR17} and \cite[Proposition~14]{LO18} for similar formulas.

\begin{theo}\label{t-quadraticform}
Let $\tau \in \mathbb{R} \setminus \{ -2, 0, 2 \}$. Then, for any $u\in\dom S_\tau$, the following holds:
\begin{align}\label{e-quadraticform}
    \|S_\tau u\|^2_{\RR^2}=\|\grad u\|^2_{\RR^2\setminus\Gamma}+m^2\|u\|^2_{\RR^2}+\frac{2m}{\tau}\|u_+|_{\Gamma}-u_-|_{\Gamma}\|^2_{\Gamma}.
\end{align}
\end{theo}

\begin{proof}
Let $u \in \dom S_\tau$ be fixed. First, as $\nu$ is pointing outside of $\Omega_+$ and inside of $\Omega_-$,  Lemma \ref{l-extension} yields
\begin{equation}\label{e-startnorm}
\begin{aligned}
    \|S_\tau u\|^2_{\RR^2}&=\|(-i\sigma\cdot\grad)u+m\sigma_3 u\|^2_{\RR^2\setminus\Gamma} \\
    &=\|(\sigma\cdot\grad)u\|^2_{\RR^2\setminus\Gamma}+m^2\|u\|^2_{\RR^2}+2\Re\left(\langle-i(\sigma\cdot\grad)u,m\sigma_3 u\rangle_{\RR^2\setminus\Gamma}\right) \\
    &=\|\grad u\|^2_{\RR^2\setminus\Gamma}+m^2\|u\|^2_{\RR^2}+2\Re\left(\langle-i(\sigma\cdot\grad)u,m\sigma_3 u\rangle_{\RR^2\setminus\Gamma}\right) \\
    &\qquad+\langle u_-|_{\Gamma},i\sigma_3 \partial_t (u_-|_{\Gamma})\rangle_{H^{1/2}(\Gamma; \mathbb{C}^2)\times H^{-1/2}(\Gamma; \mathbb{C}^2)}\\
    &\qquad -\langle u_+|_{\Gamma},i\sigma_3 \partial_t (u_+|_{\Gamma})\rangle_{H^{1/2}(\Gamma; \mathbb{C}^2)\times H^{-1/2}(\Gamma; \mathbb{C}^2)}.
\end{aligned}
\end{equation}
Recall that $M$ is defined by~\eqref{def_M} and note that $\sigma_3(\sigma\cdot\nu)=-(\sigma\cdot\nu)\sigma_3$, which holds by~\eqref{Pauli_anti_commutation}, implies
\begin{equation*}
  \sigma_3 M = \left( \frac{4+\tau^2}{4-\tau^2}\sigma_0-\frac{4\tau}{4-\tau^2}i\sigma_3(\sigma\cdot\nu) \right)  \sigma_3 = M^{-1} \sigma_3.
\end{equation*}
Since $\nu$ is piecewise constant, this yields
\begin{equation*}
\begin{aligned}
    \big\langle M u_+|_\Gamma,&i\sigma_3(\partial_t (Mu_+|_{\Gamma})\big\rangle_{H^{1/2}(\Gamma; \mathbb{C}^2)\times H^{-1/2}(\Gamma; \mathbb{C}^2)}\\
    &= \big\langle M^{-1} M u_+|_\Gamma, i\sigma_3 \partial_t (u_+|_{\Gamma})\big\rangle_{H^{1/2}(\Gamma; \mathbb{C}^2)\times H^{-1/2}(\Gamma; \mathbb{C}^2)},
\end{aligned}
\end{equation*}
cf.~\eqref{equation_commutation} for details. This yields
\begin{equation}\label{e-boundary0s}
\begin{aligned}
    \big\langle u_-|_\Gamma,&i\sigma_3\partial_t (u_-|_{\Gamma})\big\rangle_{H^{1/2}(\Gamma; \mathbb{C}^2)\times H^{-1/2}(\Gamma; \mathbb{C}^2)}
    - \big\langle u_+|_\Gamma,i\sigma_3\partial_t (u_+|_{\Gamma})\big\rangle_{H^{1/2}(\Gamma; \mathbb{C}^2)\times H^{-1/2}(\Gamma; \mathbb{C}^2)}  \\
    &= \big\langle M u_+|_\Gamma,i\sigma_3\partial_t (Mu_+|_{\Gamma})\big\rangle_{H^{1/2}(\Gamma; \mathbb{C}^2)\times H^{-1/2}(\Gamma; \mathbb{C}^2)} \\
    &\hspace{6em} -\big\langle u_+|_\Gamma,i\sigma_3\partial_t (u_+|_{\Gamma})\big\rangle_{H^{1/2}(\Gamma; \mathbb{C}^2)\times H^{-1/2}(\Gamma; \mathbb{C}^2)} \\
    &=\big\langle M^{-1} M u_+|_\Gamma,  i\sigma_3\partial_t (u_+|_{\Gamma})\big\rangle_{H^{1/2}(\Gamma; \mathbb{C}^2)\times H^{-1/2}(\Gamma; \mathbb{C}^2)} \\
    &\hspace{6em} -\big\langle u_+|_\Gamma,i\sigma_3\partial_t (u_+|_{\Gamma})\big\rangle_{H^{1/2}(\Gamma; \mathbb{C}^2)\times H^{-1/2}(\Gamma; \mathbb{C}^2)} =0. \\
\end{aligned}
\end{equation}
This simplifies equation \eqref{e-startnorm} and it remains only to recover the last term in equation \eqref{e-quadraticform}. 

The 
 identity
\begin{align*}
    \big\langle-i(\sigma\cdot\grad)u_{\pm},m\sigma_3 u_{\pm}\big\rangle\ci{\Omega_{\pm}}+\big\langle m\sigma_3 u_{\pm},-i(\sigma\cdot\grad)u_{\pm}\big\rangle\ci{\Omega_{\pm}}=\pm\big\langle-i(\sigma\cdot\nu)u_{\pm}|_\Gamma,m\sigma_3 u_{\pm}|_\Gamma\big\rangle_{\Gamma},
\end{align*}
and the fact that $M$ commutes with $-i\sigma_3(\sigma\cdot\nu)$ allow for the following simplification: 

\begin{align*}
2\Re&\left(\langle-i(\sigma\cdot\grad)u,m\sigma_3 u\rangle_{\RR^2\setminus\Gamma}\right) \\
&=\big\langle-i(\sigma\cdot\grad)u,m\sigma_3 u\big\rangle_{\RR^2\setminus\Gamma}
 +\big\langle m\sigma_3 u,-i(\sigma\cdot\grad)u\big\rangle_{\RR^2\setminus\Gamma} \\
 &=m\big[\big\langle-i\sigma_3(\sigma\cdot\nu)u_+|_\Gamma,u_+|_\Gamma\big\rangle_{\Gamma}
 -\big\langle-i\sigma_3(\sigma\cdot\nu)u_-|_\Gamma,u_-|_\Gamma\big\rangle_{\Gamma} \\
 &\hspace{8em}-\big\langle u_+|_\Gamma,-i\sigma_3(\sigma\cdot\nu)M u_+|_\Gamma\big\rangle_{\Gamma}
 +\big\langle u_+|_\Gamma,-i\sigma_3(\sigma\cdot\nu) M u_+|_\Gamma\big\rangle_{\Gamma}\Big] \\
 &=m\big[\big\langle-i\sigma_3(\sigma\cdot\nu)u_+|_\Gamma,u_+|_\Gamma\big\rangle_{\Gamma}
 -\big\langle-i\sigma_3(\sigma\cdot\nu)u_-|_\Gamma,u_-|_\Gamma\big\rangle_{\Gamma} \\
 &\hspace{8em}-\big\langle u_-|_\Gamma,-i\sigma_3(\sigma\cdot\nu)u_+|_\Gamma\big\rangle_{\Gamma}
 +\big\langle u_+|_\Gamma,-i\sigma_3(\sigma\cdot\nu)u_-|_\Gamma\big\rangle_{\Gamma}\Big] \\
 &=m\big[\big\langle-i\sigma_3(\sigma\cdot\nu)(u_+|_\Gamma-u_-|_\Gamma),(u_+|_\Gamma+u_-|_\Gamma)\big\rangle_{\Gamma}\big] \\
 &=m\left\langle-i\sigma_3(\sigma\cdot\nu)(u_+|_\Gamma-u_-|_\Gamma),-\frac{2}{\tau}i\sigma_3(\sigma\cdot\nu)(u_+|_\Gamma-u_-|_\Gamma)\right\rangle_{\Gamma}=\frac{2m}{\tau}\big\|u_+|_\Gamma-u_-|_\Gamma\big\|^2_{\Gamma},
\end{align*}
where the second to last equality is due to the transmission condition in equation \eqref{e-transmission}. The claim follows. 
\end{proof}

\section{Symmetry of  the spectrum} \label{section_symmetry_spectrum}

Recall that the distinguished self-adjoint extension of $S_\tau$ is denoted by $A_{\tau,1}$, cf. Theorems~\ref{t-nbreakdown} and~\ref{t-distinguished}. In this section we show that the spectrum of $A_{\tau,1}$ is symmetric with respect to the origin. Corollary~\ref{corollary_sigma_ess} then shows that the essential spectrum of any self-adjoint extension of $S_\tau$ is symmetric, which will be helpful in determining further spectral properties in Sections~\ref{section_sigma_ess} and~\ref{section_discrete}.

\begin{theo}\label{t-symmetry}
Let  $\omega \in (0, \frac{\pi}{2})$ and $\tau \in \mathbb{R} \setminus \{ -2, 0, 2 \}$. Then the spectrum of $A_{\tau,1}$ is symmetric about the origin, i.e.~$\la\in\sigma(A_{\tau,1})$ if and only if $-\la\in\sigma(A_{\tau,1})$. In particular, this symmetry also holds for the discrete and the essential spectrum of $A_{\tau,1}$.
\end{theo}

\begin{proof}
First, recall that the spin-orbit operator $J_\tau$ is given in equation \eqref{e-angular}. Let $\la$ be the unique simple eigenvalue of $J_\tau$ in the interval $(0,\frac{1}{2})$ from Lemma \ref{l-eigenvalues} and $\f_\la$ the corresponding eigenvector. Define
\begin{align*}
    v(x(r,\te), y(r,\te)):=K_{\la-1/2}(r)\cdot\f_\la(\te),
\end{align*}
where $x(r,\te) = r\cos\te$ and $y(r,\te) = r\sin\te$.
Then, by Theorem~\ref{t-distinguished} the distinguished self-adjoint extension of $S_\tau$ can be written as
\begin{align*}
    A_{\tau,1} u&=(-i(\sigma\cdot\grad)+m\sigma_3)u_+\oplus(-i(\sigma\cdot\grad)+m\sigma_3) u_-, \\
    \dom A_{\tau,1} &=\dom S_\tau+\spa\{v\}.
\end{align*}

Define the charge conjugation operator by $Cf:=\sigma_1\overline{f}$. We show that $C$ maps $\dom A_{\tau,1}$ into itself. Note that $\overline{\sigma_1} = \sigma_1$, $\overline{\sigma_2} = -\sigma_2$, and $\overline{\sigma_3} = \sigma_3$. First, for $u \in \dom S_\tau$, the transmission condition in equation \eqref{def_M} and the anti-commutation relations in~\eqref{Pauli_anti_commutation} yield
\begin{align*}
    0&=C\left[\left( \frac{4+\tau^2}{4-\tau^2} \sigma_0 + \frac{4\tau}{4-\tau^2}i \sigma_3 (\sigma \cdot \nu) \right) u_+|_\Gamma - u_-|_\Gamma\right]\\
    &=\sigma_1\left[\left( \frac{4+\tau^2}{4-\tau^2} \sigma_0 -\frac{4\tau}{4-\tau^2} i \sigma_3 (\sigma_1 \nu_1 - \sigma_2 \nu_2) \right) \overline{u_+|_\Gamma} - \overline{u_-|_\Gamma}\right] \\
    &=\left[\left( \frac{4+\tau^2}{4-\tau^2} \sigma_0 +\frac{4\tau}{4-\tau^2} i \sigma_3 (\sigma_1 \nu_1 + \sigma_2 \nu_2) \right) \sigma_1 \overline{u_+|_\Gamma} - \sigma_1 \overline{u_-|_\Gamma}\right] \\
    &=\left[\left( \frac{4+\tau^2}{4-\tau^2} \sigma_0 + \frac{4\tau}{4-\tau^2}i \sigma_3 (\sigma \cdot \nu) \right) C u_+|_\Gamma - C u_-|_\Gamma\right],
\end{align*}
and so $Cu\in\dom S_\tau$. 
Next, as $K_{\la-1/2}$ is real-valued and scalar, we have
\begin{align*}
    Cv(x(r,\te), y(r,\te))=K_{\la-1/2}(r)C\f_\la(\te).
\end{align*}
To show that $C\f_\la\in\ker(J_\tau-\la)$, we first verify that $C\f_{\la} \in \dom J_\tau$. As $M_l$ is a restriction of $M$, a similar consideration as above shows first $\sigma_1 \overline{M_l} = M_l \sigma_1$ and further
\begin{align*}
    C\f_{\la,-}(\w)&=\sigma_1\overline{\f_{\la,-}}(\w)=\sigma_1\overline{M_l \f_{\la,+}(\w)} = M_l\sigma_1\overline{ \f_{\la,+}(\w)} =M_l C\f_{\la,+}(\w).
\end{align*}
An analogous argument shows that $C\f_{\la,-}(2\pi-\w)=M_r C\f_{\la,+}(-\w)$, and hence $C\f_\la\in\dom J_\tau$. Moreover, 
\begin{equation*}
  \begin{split}
    J_\tau C\f_\la&=\left(-i\sigma_3\sigma_1\overline{\f_{\la,+}}'+\frac{\sigma_1\overline{\f_{\la,+}}}{2}\right)\oplus\left(-i\sigma_3\sigma_1\overline{\f_{\la,-}}'+\frac{\sigma_1\overline{\f_{\la,-}}}{2}\right) \\
    &=\sigma_1\left[\left(i\sigma_3\overline{\f_{\la,+}'}+\frac{\overline{\f_{\la,+}}}{2}\right)\oplus\left(i\sigma_3\overline{\f_{\la,-}'}+\frac{\overline{\f_{\la,-}}}{2}\right)\right] \\
    &=\sigma_1\overline{\left[\left(-i\sigma_3\f_{\la,+}'+\frac{\f_{\la,+}}{2}\right)\oplus\left(-i\sigma_3\f_{\la,-}'+\frac{\f_{\la,-}}{2}\right)\right]} \\
    &=\sigma_1\overline{\la\f_\la}=\la C\f_\la.
  \end{split}
\end{equation*}
Since $\la$ is a simple eigenvalue of $J_\tau$ and $\f_\la, C\f_\la\in\ker(J_\tau-\la)$, we find that $C\f_\la=a\f_\la$ for some $a\in\CC$. Hence, $Cv\in\spa\{v\}$ and for any $u\in\dom A_{\tau,1}$, $Cu\in\dom A_{\tau,1}$.
Finally, using $\sigma_2=-\overline{\sigma_2}$, we calculate for $u\in\dom A_{\tau,1}$ 
\begin{align*}
    (A_{\tau,1}  Cu)_\pm&=[(-i(\sigma\cdot\grad)+m\sigma_3]Cu_\pm=[-i(\sigma\cdot\grad)+m\sigma_3]\sigma_1\overline{u}_\pm \\
    &=\sigma_1(-i\sigma_1\partial_1+i\sigma_2\partial_2-m\sigma_3)\overline{u}_\pm=-\sigma_1\overline{(-i\sigma_1\partial_1-i\sigma_2\partial_2+m\sigma_3)u_\pm} \\
    &=-C(-i(\sigma\cdot\grad)+m\sigma_3)u_\pm=-C(A_{\tau,1}  u)_\pm.
\end{align*}
If we let $\la\in\sigma(A_{\tau,1})$ and $(u_n)_{n=1}^{\infty}\subset\dom A_{\tau,1}$ such that $\|u_n\|_{\mathbb{R}^2}=1$ and $(A_{\tau,1} -\la)u_n\to 0$, then $Cu_n\in\dom A_{\tau,1}$, $\|Cu_n\|_{\mathbb{R}^2}=1$ and 
\begin{align*}
    (A_{\tau,1} +\la)Cu_n=C(-A_{\tau,1} +\la)u_n\to0,
\end{align*}
so that $-\la\in\sigma(A_{\tau,1} )$. The result follows. 
\end{proof}

\begin{cor} \label{corollary_sigma_ess}
Let $\omega \in (0, \frac{\pi}{2})$ and $\tau \in \mathbb{R} \setminus \{ -2, 0, 2 \}$. Let $z \in \mathbb{C}$ with $|z|=1$ and $A_{\tau,z}$ be the corresponding self-adjoint extension of $S_\tau$ via Theorem~\ref{t-nbreakdown}. Then the essential spectrum of $A_{\tau,z}$ is symmetric with respect to the origin, i.e.~$\sigma\ti{ess}(A_{\tau,z})=-\sigma\ti{ess}(A_{\tau,z})$.
\end{cor}

\begin{proof}
The claim holds for the distinguished self-adjoint extension $A_{\tau,1}$ by Theorem \ref{t-symmetry} and since 
\begin{align*}
    (A_{\tau,z}-i)^{-1}-(A_{\tau,1}-i)^{-1}
\end{align*}
has finite rank, it follows that $\sigma\ti{ess}(A_{\tau,z})=\sigma\ti{ess}(A_{\tau,1})$.
\end{proof}

\section{Essential Spectrum} \label{section_sigma_ess}

In this section we show the statements about the essential spectrum of any self-adjoint extension of $S_\tau$ in Theorem~\ref{theorem_intro}. We begin with a general result for symmetric operators. Then, in Proposition~\ref{proposition_essential_spectrum1} we show that $(-\infty, -m] \cup [m,\infty)\subset\sigma\ti{ess}(A_{\tau,z})$ holds for any self-adjoint extension $A_{\tau,z}$ of $S_\tau$. Together with Theorem~\ref{t-quadraticform} and Proposition~\ref{p-vladnote} this implies already that this inclusion is an equality if $\tau \in (0, \infty) \setminus \{ 2 \}$. Then, in Propositions~\ref{p-essspec1} and~\ref{p-essspec2} we show with a Neumann bracketing argument and suitable singular sequences that for $\tau \in (-\infty, 0) \setminus \{-2\}$ one has
\begin{equation*}
  \sigma\ti{ess}(S_\tau^*S_\tau) = \Bigg[m^2\left(\frac{\tau^2-4}{\tau^2+4}\right)^2,\infty\Bigg).
\end{equation*}
This allows us finally with Proposition~\ref{p-vladnote} and Corollary~\ref{corollary_sigma_ess} to determine $\sigma_\textup{ess}(A_{\tau,z})$ also for $\tau \in (-\infty, 0) \setminus \{-2\}$.

\begin{prop}\label{p-vladnote}
Let $S$ be a closed, densely defined, symmetric operator in the Hilbert space $\cH$ with deficiency indices $(n,n)$, $n\in\NN$, and $D$ be a self-adjoint extension of $S$.
Then the resolvent difference
\[
	W:=(S^*S+1)^{-1} - (D^2+1)^{-1}
\]
is a finite-rank operator with rank at most $n$.
\end{prop}

\begin{proof}
The self-adjoint extension $D$ can be described by the von Neumann extension theory of symmetric operators. According to~\cite[Theorem 13.10]{S12} there exists a unitary operator 
$U:\ker(S^*-i)\to\ker(S^*+i)$
such that the self-adjoint extension $D$ is defined as
\[
	D u = S u_0 + i \wt{u} -iU\wt{u},\qquad\dom D = \{u_0+\wt{u}+U\wt{u}: u_0\in\dom S,\,\wt{u}\in\ker(S^*-i)\}. 
\]
Likewise, \cite[Proposition 3.7]{S12} says that $u\in\dom S^*$ admits a unique decomposition $u = u_0 + \wt{u} +\wh{u}$ for $u_0\in\dom S$, $\wt{u}\in\ker(S^*-i)$ and $\wh{u}\in\ker(S^*+i)$. Define the linear operator $\Pi:\dom S^*\to\ker(S^*-i)$ such that $\Pi u := \wt{u}$ and let $f,g\in\cH$ be arbitrary. We set
\[
	u:= (S^*S+1)^{-1}f,\qquad v:=(D^2+1)^{-1}g,
\]
which, in particular, means that $u\in\dom(S^* S)$ and $v\in\dom(S^*)$. We calculate that 
\[
\begin{aligned}
	(W f,g)_{\cH} &=
	((S^*S+1)^{-1}f,g)_\cH - (f,(D^2+1)^{-1}g)_\cH\\
	&=
	(u,D^2 v)_\cH - (S^*S u,v)_\cH\\ 
	&=(u,S^*S v_0 - \wt{v} -U\wt{v})_\cH - (S^*S u, v_0 + \wt{v}+U\wt{v})_\cH\\
	&= (u,-\wt{v}-U\wt{v})_\cH - (S^*S u,\wt{v}+U\wt{v})_\cH \\
	& = -((S^*S+1)u,\wt{v}+U\wt{v})_\cH\\
	&=
	-(f,(1+U)\Pi(D^2+1)^{-1}g)_\cH,
\end{aligned}
\]
where the symmetry of $S^*S$ was used to cancel two terms between the third and fourth equalities. 
Hence,
\[
	W = -(1+U)\Pi(D^2+1)^{-1}
\]
and because $\ran \Pi = \ker(S^*-i)$, the rank of $W$ can be at most $n$.
\end{proof}

In the next proposition we show that the inclusion $(-\infty, -m] \cup [m,\infty)\subset\sigma\ti{ess}(A_{\tau,z})$ holds for any self-adjoint extension $A_{\tau,z}$ of $S_\tau$.

\begin{prop} \label{proposition_essential_spectrum1}
Let $A_{\tau,z}$ be any self-adjoint extension of $S_\tau$ defined in~\eqref{def_S_tau}. Then the inclusion $(-\infty, -m] \cup [m,\infty)\subset\sigma\ti{ess}(A_{\tau,z})$ holds.
\end{prop}

\begin{proof}
Let $\la\in(-\infty, -m) \cup (m, \infty)$ be fixed and fix a vector $\zeta\in\CC^2\setminus\{0\}$ such that $(\sqrt{\la^2-m^2}\sigma_1+m\sigma_3+\la \sigma_0)\zeta\neq0$. Also let $\chi\in C_0^{\infty}(\RR)$ be chosen so that $\chi(r)=1$ for $|r|<1/2$, $\chi(r)=0$ for $|r|>1$, and $x_n=(-1-n^2,0)^T$. Now, consider the sequence 
\begin{align}\label{e-weylsequence}
\psi_n(x)=\frac{1}{n}\chi\left(\frac{1}{n}|x-x_n|\right)e^{i\sqrt{\la^2-m^2}x\cdot e_1}\left(\sqrt{\la^2-m^2}\sigma_1+m\sigma_3+\la \sigma_0 \right)\zeta,\qquad n\in\N,
\end{align}
where $e_1 = (1,0)^T$.
As in \cite[Theorem 5.7(i)]{BH}, it is possible to show that 
\begin{itemize}
    \item[(i)] $\psi_n\in \dom S_\tau$,
    \item[(ii)] $\|\psi_n\|^2_{\RR^2} = c>0$ for all $n$ with a constant $c$ independent of $n$,
    \item[(iii)] $\psi_n$ converges weakly to zero and,
    \item[(iv)] $(S_\tau-\la)\psi_n\to0$ as $n\to\infty$.
\end{itemize}
The sequence $\{\psi_n\}_{n=1}^{\infty}$ is thus a Weyl sequence for $\la$ and any self-adjoint extension $A_{\tau,z}$ of $S_\tau$, which implies $(-\infty, -m) \cup (m,\infty)\subset\sigma\ti{ess}(A_{\tau,z})$. Since the essential spectrum of any self-adjoint operator is always closed, the result follows.
\end{proof}

Proposition~\ref{proposition_essential_spectrum1} can now be combined with Theorem~\ref{t-quadraticform} to characterize the essential spectrum for any positive interaction strength $\tau$ such that $\tau\neq 2$.

\begin{cor} \label{corollary_sigma_ess_pos}
  Let $\tau \in (0, \infty) \setminus \{2\}$ and $A_{\tau,z}$ be any self-adjoint extension of $S_\tau$, defined in~\eqref{def_S_tau}. Then $\sigma\ti{ess}(A_{\tau,z}) = (-\infty, -m] \cup [m,\infty)$.
\end{cor}
\begin{proof}
 Propositions~\ref{p-vladnote} and~\ref{proposition_essential_spectrum1} immediately imply that $[m^2, \infty) \subset \sigma_\textup{ess}(S_\tau^* S_\tau)$. Moreover, Theorem~\ref{t-quadraticform} says that for any $u \in \dom (S_\tau^* S_\tau)$
  \begin{equation*}
    \begin{split} 
      (S_\tau^* S_\tau u, u)_{\mathbb{R}^2} &= \|S_\tau u\|^2_{\RR^2}=\|\grad u\|^2_{\RR^2\setminus\Gamma}+m^2\|u\|^2_{\RR^2}+\frac{2m}{\tau} \big\|u_+|_\Gamma-u_-|_\Gamma\big\|^2_{\Gamma} 
      \geq m^2 \|u\|^2_{\RR^2},
    \end{split}
  \end{equation*}
  which yields the reverse inclusion $\sigma_\textup{ess}(S_\tau^* S_\tau) \subset [m^2, \infty)$. Hence, $\sigma_\textup{ess}(S_\tau^* S_\tau) = [m^2, \infty)$. Proposition~\ref{p-vladnote} says that $S^*_\tau S_\tau$ and $(A_{\tau,z})^2$ differ by a finite-rank perturbation and so their essential spectra are equal. Corollary~\ref{corollary_sigma_ess} further shows that the essential spectrum of $A_{\tau,z}$ is symmetric about $0$ and the result follows.
\end{proof}

We now turn our attention to the case where $\tau$ is negative such that $\tau\neq -2$. The corresponding claim about the essential spectrum of any self-adjoint extension of $S_\tau$ from Theorem \ref{theorem_intro} will be shown in three steps. Begin by setting $\varepsilon_\tau := m \frac{\tau^2-4}{\tau^2+4}$. First, a Neumann bracketing argument is used to prove that $[\varepsilon_\tau^2, \infty) \subset \sigma_\textup{ess}(S_\tau^* S_\tau)$. The reverse inclusion is contained in Proposition \ref{p-essspec2} and uses a particular Weyl sequence construction. The final statements then follow by a similar argument as the proof of Corollary \ref{corollary_sigma_ess_pos}.

\begin{prop}\label{p-essspec1}
Let $\tau \in (-\infty, 0) \setminus \{ -2 \}$. Then
\begin{align*}
    \sigma\ti{ess}(S_\tau^*S_\tau) \subset \Bigg[m^2\left(\frac{\tau^2-4}{\tau^2+4}\right)^2,\infty\Bigg).
\end{align*}
\end{prop}

\begin{proof}
As indicated in Figure \ref{f-neumannbracketing}, we define for a fixed $\gamma > 0$ the following domains $\Omega_j \subset \mathbb{R}^2$, $j \in \{ 1, \dots, 5 \}$:
\begin{equation*}
  \begin{split}
    \Omega_1 &:= \left\{ \left(\frac{\gamma}{\sin \omega} + r \cos \varphi, r \sin \varphi \right): r > 0, \varphi \in (-\omega, \omega) \right\}, \\
    \Omega_2 &:= \left\{ \left(-\frac{\gamma}{\sin \omega} + r \cos \varphi, r \sin \varphi \right): r > 0, \varphi \in (\omega, 2 \pi - \omega) \right\}, \\
    \Omega_3 &:= \left\{ \left(r \cos \omega + s \sin \omega, r \sin \omega - s \cos \omega \right): r > \frac{\gamma}{\tan \omega}, s \in (-\gamma, \gamma) \right\}, \\
    \Omega_4 &:= \left\{ \left(r \cos \omega + s \sin \omega, -r \sin \omega + s \cos \omega \right): r > \frac{\gamma}{\tan \omega}, s \in (-\gamma, \gamma) \right\}, \\
    \Omega_5 &:= \mathbb{R}^2 \setminus \overline{\big( \Omega_1 \cup \Omega_2 \cup \Omega_3 \cup \Omega_4 \big)}.
  \end{split}
\end{equation*}

\begin{figure}[h]
\centering
\includegraphics[scale=1]{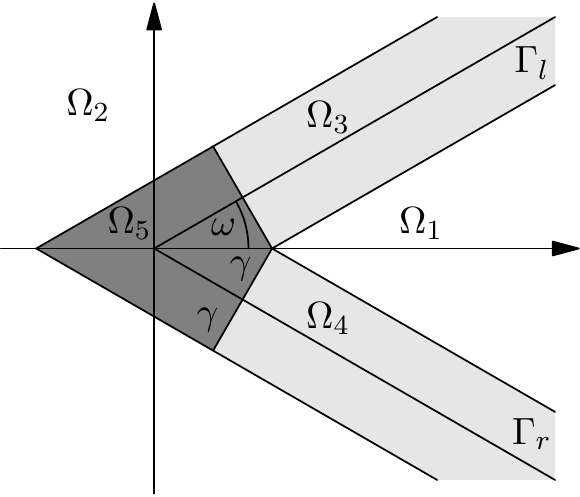}
\caption{Decomposition of $\RR^2$ for Neumann bracketing. The domains $\Omega_3$ and $\Omega_4$ are depicted in light gray, the set $\Omega_5$ in dark gray.}
\label{f-neumannbracketing}
\end{figure}

In each of the regions $\Omega_j$, we define a quadratic form in $L^2(\Omega_j)$: for $j=1,2$ by
\begin{align*}
    \mathfrak{b}_j[u]&=\|\grad u\|^2\ci{\Omega_j} + m^2\|u\|^2\ci{\Omega_j},\quad\dom \mathfrak{b}_j = H^1(\Omega_j; \CC^2),
\end{align*}
and for $j=3,4,5$ with $\Gamma_j := \Gamma \cap \Omega_j$ by
\begin{align*}
    \mathfrak{b}_j[u]&=\|\grad u\|^2_{\Omega_j\setminus\Gamma_j}+m^2\|u\|^2_{\Omega_j}+\frac{2m}{\tau}\|u_+|_{\Gamma_j}-u_-|_{\Gamma_j}\|_{\Gamma_j}^2, \\
    \dom \mathfrak{b}_j &=\big\{u\in H^1(\Omega_j\setminus\Gamma_j;\CC^2)~:~u_-|_{\Gamma_j}=Mu_+|_{\Gamma_j}\big\},
\end{align*}
where functions $u \in L^2(\Omega_j; \mathbb{C}^2)$ are written with the notation $u_\pm = u \upharpoonright (\Omega_j \cap \Omega_\pm)$. The quadratic forms $\mathfrak{b}_j$, $j \in \{ 1, \dots, 5 \}$, are closed and semibounded from below. The associated self-adjoint operators are denoted by $B_j$. If we define $T=\bigoplus_{j=1}^5 B_j$, then we have, in the sense of semibounded self-adjoint operators, $S_\tau^* S_\tau \geq T$; i.e. for any $u \in \dom S_\tau$ there holds $u\in\dom\mathfrak{t}$ and $\|S_\tau u\|_{\RR^2}^2 \ge \mathfrak{t}[u]$, where $\mathfrak{t}$ is the quadratic
form in $L^2(\RR^2;\CC^2)$ associated with the operator $T$.  The min-max principle implies
\begin{equation} \label{equation_inf_sigma_ess}
  \inf \sigma_\textup{ess}(S_\tau^* S_\tau) \geq \inf \sigma_\textup{ess}(T).
\end{equation}

Begin by observing that $B_1$ and $B_2$ are shifted Neumann Laplacians in $\Omega_1$ and $\Omega_2$, respectively, and hence 
\begin{equation} \label{estimate_B_1}
  B_1,B_2\geq m^2.
\end{equation} 
Moreover, as $\Omega_5$ is a bounded Lipschitz domain, $\dom \mathfrak{b}_5$ is compactly embedded in $L^2(\Omega_5; \mathbb{C}^2)$ and thus, 
\begin{equation} \label{estimate_B_5}
  \sigma_\textup{ess}(B_5)=\emptyset.
\end{equation}

It therefore remains to estimate $\inf \sigma_\textup{ess}(B_j)$, $j \in \{ 3, 4 \}$. Let $\Omega_\gamma:=\RR_+\times(-\gamma, \gamma)$ and
\begin{align*}
    M_j=\frac{4+\tau^2}{4-\tau^2}\sigma_0+\frac{4\tau}{4-\tau^2}i\sigma_3 (\sigma \cdot \nu_j), \qquad \nu_j = \nu \upharpoonright (\Gamma \cap \Omega_j).
\end{align*}
Define $\widetilde{\Gamma} := \mathbb{R}_+ \times \{ 0 \}$.
Then, a rotation by $\pm \omega$, translation by $(-\frac{\gamma}{\tan \omega}, 0)$, and finally for $\Omega_3$ a change of variables $y \rightarrow -y$ shows that $B_j$, $j\in\{3,4\}$, is unitarily equivalent to a self-adjoint operator $\widetilde{B}_j$ associated with the closed and semibounded quadratic form
\begin{align*}
    \wt{\mathfrak{b}}_j[u]&=\|\grad u\|^2_{\Omega_\gamma \setminus \widetilde{\Gamma}}+m^2\|u\|^2_{\Omega_\gamma}+\frac{2m}{\tau}\big\|u_+|_{\widetilde{\Gamma}}-u_-|_{\widetilde{\Gamma}} \big\|_{\widetilde{\Gamma}}^2, \\
    \dom \wt{\mathfrak{b}}_j &= \big\{u\in H^1(\Omega_\gamma \setminus \widetilde{\Gamma};\CC^2)\colon u_-|_{\widetilde{\Gamma}}=M_j u_+|_{\widetilde{\Gamma}}\big\},
\end{align*}
where $u_+ = u|_{ {\RR_+\times(0,\gamma)}}$ and $u_- =
	u|_{\RR_+\times(-\gamma,0)}$.
Set $\Theta_j := \frac{1}{\sqrt{2}} (\sigma_0 + i (\sigma \cdot \nu_j))$. A direct calculation shows that $\Theta_j$ is a unitary matrix and that
\begin{equation} \label{def_wt_M}
  \Theta_j^* M_j \Theta_j = \frac{4+\tau^2}{4-\tau^2}\sigma_0-\frac{4\tau}{4-\tau^2}\sigma_3 =: \wt{M}.
\end{equation}
We define for $j\in\{3,4\}$
\[
\mathfrak{c}[u] := \wt{\mathfrak{b}}_j[\Theta_j u], 
\qquad\dom\mathfrak{c} :=
\Theta_j^*\big(\dom \widetilde{\mathfrak{b}}_j\big),\]
and note that this form is independent of $j$ and explicitly given by
\begin{align*}
    \mathfrak{c}[u]&=\|\nabla u\|^2_{\Omega_\gamma \setminus\widetilde{\Gamma}}+m^2\|u\|^2_{\Omega_\gamma}+\frac{2m}{\tau} \big\|u_+|_{\widetilde{\Gamma}}-u_-|_{\widetilde{\Gamma}} \big\|_{\widetilde{\Gamma}}^2, \\
    \dom \mathfrak{c} &=\left\{u\in H^1(\Omega_\gamma \setminus \widetilde{\Gamma};\CC^2)\colon u_-|_{\widetilde{\Gamma}}=\wt{M} u_+|_{\widetilde{\Gamma}}\right\}.
\end{align*}
Since $\Theta_j$ is unitary, we have for $j \in \{ 3, 4 \}$ 
\begin{equation} \label{estimate_sigma_ess}
  \inf \sigma_\textup{ess}(B_{j}) \geq \inf \left\{ \frac{\mathfrak{c}[u]}{\|u\|_{\Omega_\gamma}^2}\colon u \in \dom \mathfrak{c}\setminus\{0\} \right\}.
\end{equation}
Let $u \in \dom \mathfrak{c}$. Then, it follows in a similar way as in the proof of \cite[Lemma~4.13]{HOBP}, cf. also Appendix~\ref{appendix_auxiliary_problem} for details, for almost every $x > 0$ that
\begin{equation} \label{estimate_lower_bound}
  \begin{split}
    \int_{(-\gamma, \gamma) \setminus \{ 0 \}} &\big( |\partial_x u(x,y)|^2 + |\partial_y u(x,y)|^2 \big) dy + m^2 \int_{(-\gamma, \gamma)} |u(x,y)|^2 dy +\frac{2m}{\tau} |u(x, 0^+)-u(x, 0^-)|^2 \\
    &\geq \int_{(-\gamma, \gamma) \setminus \{ 0 \}} |\partial_y u(x,y)|^2 dy + m^2 \int_{(-\gamma, \gamma)} |u(x,y)|^2 dy +\frac{2m}{\tau} |u(x, 0^+)-u(x, 0^-)|^2  \\
    &\geq E(\gamma) \int_{(-\gamma, \gamma)} |u(x,y)|^2 dy,
  \end{split}
\end{equation}
where $E(\gamma) \in [0, m^2)$ such that
\begin{align*}
    \lim_{\gamma \rightarrow \infty} E(\gamma) = m^2\left(\frac{\tau^2-4}{\tau^2+4}\right)^2.
\end{align*}
This implies that
\begin{equation*}
  \begin{split}
    \mathfrak{c}[u] &= \int_{\mathbb{R}_+} \bigg( \int_{(-\gamma, \gamma) \setminus \{ 0 \}} \big(|\partial_x u(x,y)|^2 + |\partial_y u(x,y)|^2\big) dy \\
    &\qquad\qquad \qquad + m^2 \int_{(-\gamma, \gamma)} |u(x,y)|^2 dy +\frac{2m}{\tau} |u(x, 0^+)-u(x, 0^-)|^2\bigg) dx \\
    &\geq \int_{\mathbb{R}_+} E(\gamma) \int_{(-\gamma, \gamma)} |u(x,y)|^2 dy dx.
  \end{split}
\end{equation*}
Therefore, equations \eqref{equation_inf_sigma_ess}--\eqref{estimate_B_5} and~\eqref{estimate_sigma_ess} imply that with $\gamma \rightarrow \infty$ we have $\inf\sigma\ti{ess}(S_\tau^* S_\tau)\geq m^2\big(\frac{\tau^2-4}{\tau^2+4}\big)^2$.
\end{proof}

We now prove the opposite inclusion.

\begin{prop}\label{p-essspec2}
Let $\tau \in (-\infty, 0) \setminus \{ -2 \}$. Then
\begin{align*}
    \Bigg[m^2\left(\frac{\tau^2-4}{\tau^2+4}\right)^2,\infty\Bigg)\subset\sigma\ti{ess}(S_\tau^*S_\tau).
\end{align*}
\end{prop}

\begin{proof}
Set $z=-m \frac{4 \tau}{\tau^2+4}$, $x_n=n^2+K$ for $n\in\NN$ and some $K>0$, $\lambda > 0$, and choose $\chi\in C_0^{\infty}(\RR)$ such that $\chi(r)=1$ for $|r|\leq 1/2$ and $\chi(r)=0$ for $|r|>1$. Define the function
\begin{align*}
    \psi_n(x,y)&:=\frac{1}{\sqrt{n}}\chi\left(\frac{\xi-x_n}{n}\right)\chi\left(\frac{c \zeta}{n}\right)e^{i\la \xi} v(\zeta), \\
    &\text{ with }\quad v(\zeta)=\begin{cases}
    a e^{-z\zeta}, & \zeta>0, \\
    M_l ae^{z \zeta}, &\zeta<0,
    \end{cases} \qquad \begin{pmatrix} \xi \\ \zeta \end{pmatrix} = \begin{pmatrix} \cos \omega & \sin \omega \\ \sin \omega &-\cos \omega \end{pmatrix} \begin{pmatrix} x \\ y \end{pmatrix},
\end{align*}
where $0 \neq a\in\CC^2$, $M_l$ is the matrix in~\eqref{def_M_l_r} and $c >0$ is chosen so that
$\supp(\psi_n)\cap\Gamma_r=\emptyset$. We will show that
\begin{enumerate}
    \item[(i)] $\psi_n\in\dom S_\tau$ for all $n$,
    \item[(ii)] $c_1 \leq \|\psi_n\|^2_{\RR^2}\leq c_2$ for constants $c_2\ge c_1>0$,
    \item[(iii)] $\psi_n\perp\psi_m$ for $m\neq n$, so that $\psi_n$ converges weakly to zero and,
    \item[(iv)] there exists a constant $c_3 >0$ such that for any $u\in\dom S_\tau$, 
    \begin{align*}
        \left|\left\langle\left(S_\tau+\sqrt{m^2-z^2+\la^2}\right)u,\left(S_\tau-\sqrt{m^2-z^2+\la^2}\right)\psi_n \right\rangle_{\RR^2}\right|\leq\frac{c_3}{n} \left( \| S_\tau u \|^2_{\mathbb{R}^2} + \| u \|^2_{\mathbb{R}^2} \right)^{1/2}.
    \end{align*}
   \end{enumerate}
Combined, these facts show that $\psi_n$ satisfies Weyl's criterion for quadratic forms in the sense of \cite[Appendix]{KL} and $m^2-z^2+\la^2 =m^2 \big(\frac{\tau^2-4}{\tau^2+4}\big)^2 + \lambda^2 \in\sigma\ti{ess}(S_\tau^* S_\tau)$. The choice of $\la>0$ was arbitrary and the essential spectrum is always closed, so the result follows. 

It remains to show claims (i)--(iv) above. Item (i) is clear by construction. To see claim (ii), we compute
\begin{equation*} 
  \begin{split}
    \| \psi_n \|_{\mathbb{R}^2}^2 &= \int_{\mathbb{R}^2} \frac{1}{n} \left| \chi\left(\frac{\xi-x_n}{n}\right)\chi\left(\frac{c \zeta}{n}\right)\right|^2 | v(\zeta)|^2 d x dy \\
    &= \int_{-1}^1 |\chi(x')|^2 dx' \int_{-n}^n \left|\chi\left(\frac{c \zeta}{n}\right)\right|^2 | v(\zeta)|^2 d \zeta,
  \end{split}
\end{equation*}
where the substitution $x' = \frac{\xi-x_n}{n}$ was done. The properties of the function $v$ then imply that claim (ii) is true.

In light of claim  (ii), it is easy to see that claim (iii) holds true, as the functions $\psi_n$ and $\psi_m$ for $n\ne m$ have disjoint supports.

Claim (iv) makes use of Theorem~\ref{t-quadraticform} and the transformation from the variables $(x,y)$ to $(\xi, \zeta)$. Let $u \in \dom S_\tau$ and set $\tilde{u}(\xi, \zeta) := u(x(\xi, \zeta), y(\xi, \zeta))$, $\tilde{\psi}_n(\xi, \zeta) := \psi_n(x(\xi, \zeta), y(\xi, \zeta))$, and $\widetilde{\Gamma} := \mathbb{R}_+ \times \{ 0 \}$. These facts allow for the definition 
\begin{equation*}
  \begin{split}
    I := \bigg\langle&\left(S_\tau+\sqrt{m^2-z^2+\la^2}\right)u,\left(S_\tau-\sqrt{m^2-z^2+\la^2}\right)\psi_n \bigg\rangle_{\RR^2} 
    = \left(z^2-\lambda^2\right) (\tilde{u}, \tilde{\psi}_n)_{\RR^2}  \\
    &+(\grad \tilde{u}, \grad \tilde{\psi}_n)_{\RR^2\setminus\widetilde{\Gamma}}
    +\frac{2m}{\tau} \int_{x_n-n}^{x_n+n} (\tilde{u}_+|_{\widetilde{\Gamma}}(\xi) - \tilde{u}_-|_{\widetilde{\Gamma}}(\xi)) \overline{(\tilde{\psi}_{n, +}|_{\widetilde{\Gamma}}(\xi) - \tilde{\psi}_{n, -}|_{\widetilde{\Gamma}}(\xi))} d \xi.
  \end{split}
\end{equation*}
Note that $u, \psi_n \in \dom S_\tau$ implies $\tilde{u}_-|_{\widetilde{\Gamma}}(\xi) = M_l \tilde{u}_+|_{\widetilde{\Gamma}}(\xi)$, $\tilde{\psi}_{n, -}|_{\widetilde{\Gamma}}(\xi) = M_l \tilde{\psi}_{n,+}|_{\widetilde{\Gamma}}(\xi)$ and $M_l=M_l^*$. Integration by parts then yields
\begin{equation} \label{equation_singular_sequence1}
  \begin{split}
    I &= (\tilde{u}, -\Delta \tilde{\psi}_n)_{\RR^2\setminus\widetilde{\Gamma}} 
    +\left(z^2-\lambda^2\right) (\tilde{u}, \tilde{\psi}_n)_{\RR^2} \\ 
    &+ \int_{x_n-n}^{x_n+n} \tilde{u}_+|_{\widetilde{\Gamma}}(\xi) \overline{\left( \frac{2m}{\tau} (\sigma_0 - M_l)^2 \tilde{\psi}_{n, +}|_{\widetilde{\Gamma}}(\xi) + \left( M_l \frac{\partial \tilde{\psi}_{n, -}}{\partial \zeta}\bigg|_{\widetilde{\Gamma}}(\xi) - \frac{\partial \tilde{\psi}_{n, +}}{\partial \zeta}\bigg|_{\widetilde{\Gamma}}(\xi) \right) \right)} d \xi.
  \end{split}
\end{equation}
We then calculate
\begin{align*}
    \frac{\partial \tilde{\psi}_{n, +}}{\partial \zeta}(\xi, \zeta) &= 
    \frac{c}{n^{3/2}} \chi\left(\frac{\xi - x_n}{n} \right) \chi'\left( \frac{c \zeta}{n} \right) e^{i \lambda \xi} a e^{-z \zeta} \\
    & \hspace{12em}- \frac{z}{n^{1/2}} \chi\left(\frac{\xi - x_n}{n} \right) \chi\left( \frac{c \zeta}{n} \right) e^{i \lambda \xi} a e^{-z \zeta},  \\
    \frac{\partial \tilde{\psi}_{n, -}}{\partial \zeta}(\xi, \zeta) &= 
    \frac{c}{n^{3/2}} \chi\left(\frac{\xi - x_n}{n} \right) \chi'\left( \frac{c \zeta}{n} \right) e^{i \lambda \xi} M_l a e^{z \zeta} \\
    & \hspace{12em}+ \frac{z}{n^{1/2}} \chi\left(\frac{\xi - x_n}{n} \right) \chi\left( \frac{c \zeta}{n} \right) e^{i \lambda \xi} M_l a e^{z \zeta}.
\end{align*}
Using that $\chi(0) = 1$ and $\chi'(0) = 0$ we infer 
\begin{equation*}
 \begin{split}
    M_l \frac{\partial \tilde{\psi}_{n, -}}{\partial \zeta}\big|_{\widetilde{\Gamma}}(\xi) - \frac{\partial \tilde{\psi}_{n, +}}{\partial \zeta}\big|_{\widetilde{\Gamma}}(\xi) = \frac{z}{n^{1/2}} \chi\left(\frac{\xi - x_n}{n} \right) e^{i \lambda \xi} (M_l^2 + \sigma_0) a = z (M_l^2 + \sigma_0) \tilde{\psi}_{n,+}|_{\widetilde{\Gamma}}(\xi). 
\end{split}
\end{equation*}
A direct calculation shows that $z (M_l^2 + \sigma_0) = -\frac{8m \tau}{4-\tau^2} M_l$ and $\frac{2m}{\tau} (\sigma_0-M_l)^2 = \frac{8m \tau}{4-\tau^2} M_l$. Hence,
we obtain that
\begin{equation} \label{equation_singular_sequence2}
    \int_{x_n-n}^{x_n+n} \tilde{u}_+|_{\widetilde{\Gamma}}(\xi) \overline{\left( \frac{2m}{\tau} (\sigma_0 - M_l)^2 \tilde{\psi}_{n, +}|_{\widetilde{\Gamma}}(\xi) + \left( M_l \frac{\partial \tilde{\psi}_{n, -}}{\partial \zeta}\bigg|_{\widetilde{\Gamma}}(\xi) - \frac{\partial \tilde{\psi}_{n, +}}{\partial \zeta}\bigg|_{\widetilde{\Gamma}}(\xi) \right) \right)} d \xi  =0.
\end{equation}
Returning to the analysis of other terms in $I$, we calculate
\begin{equation*}
  \begin{split}
    \Delta \tilde{\psi}_n(\xi, \zeta) &= (z^2 - \lambda^2) \tilde{\psi}_n(\xi, \zeta) + \frac{1}{n^{5/2}} \chi''\left( \frac{\xi-x_n}{n} \right) \chi\left( \frac{c \zeta}{n} \right) e^{i \lambda \xi} v(\zeta) \\
    &\quad + \frac{2i \lambda}{n^{3/2}} \chi'\left( \frac{\xi-x_n}{n} \right) \chi\left( \frac{c \zeta}{n} \right) e^{i \lambda \xi} v(\zeta) + \frac{c^2}{n^{5/2}} \chi\left( \frac{\xi-x_n}{n} \right) \chi''\left( \frac{c \zeta}{n} \right) e^{i \lambda \xi} v(\zeta) \\
    & \quad +\frac{2c}{n^{3/2}} \chi\left( \frac{\xi-x_n}{n} \right) \chi'\left( \frac{c \zeta}{n} \right) e^{i \lambda \xi} v'(\zeta).
  \end{split}
\end{equation*}
The Cauchy-Schwarz inequality yields
\begin{equation} \label{equation_singular_sequence3}
  \begin{split}
    \bigg|&(\tilde{u}, -\Delta \tilde{\psi}_n)_{\RR^2\setminus\widetilde{\Gamma}} +\left(z^2-\lambda^2\right) (\tilde{u}, \tilde{\psi}_n)_{\RR^2} \bigg| \\
    &= \bigg|\int_{\mathbb{R}^2} \tilde{u}(\xi, \zeta) \overline{\bigg(\frac{1}{n^{5/2}} \chi''\left( \frac{\xi-x_n}{n} \right) \chi\left( \frac{c \zeta}{n} \right) e^{i \lambda \xi} v(\zeta) + \frac{2i \lambda}{n^{3/2}} \chi'\left( \frac{\xi-x_n}{n} \right) \chi\left( \frac{c \zeta}{n} \right) e^{i \lambda \xi} v(\zeta)} \\
    &\quad \qquad \overline{+ \frac{c^2}{n^{5/2}} \chi\left( \frac{\xi-x_n}{n} \right) \chi''\left( \frac{c \zeta}{n} \right) e^{i \lambda \xi} v(\zeta) +\frac{2c}{n^{3/2}} \chi\left( \frac{\xi-x_n}{n} \right) \chi'\left( \frac{c \zeta}{n} \right) e^{i \lambda \xi} v'(\zeta) \bigg)} d \xi d \zeta \bigg| \\
    &\leq \frac{C}{n} \| u \|_{\mathbb{R}^2}.
  \end{split}
\end{equation}
The absolute value of the expression $I$ in equation \eqref{equation_singular_sequence1} can thus be estimated by equations \eqref{equation_singular_sequence2} and \eqref{equation_singular_sequence3} and we have
\begin{equation} \label{equation_singular_sequence4}
  |I| \leq \frac{C}{n} \| u \|_{\mathbb{R}^2} \leq \frac{C}{n}\left(\|S_\tau u\|^2_{\RR^2}+\|u\|^2_{\RR^2}\right)^{1/2},
\end{equation}
which is claim (iv). The proof of the proposition is therefore complete.
\end{proof}

In analogy to Corollary~\ref{corollary_sigma_ess_pos}, it is possible to combine Corollary~\ref{corollary_sigma_ess} and  Propositions~\ref{p-essspec1} and~\ref{p-essspec2} in order to show the following result: 
\begin{cor} \label{corollary_sigma_ess_neg}
  Let $\tau \in (-\infty, 0) \setminus \{-2\}$ and let $A_{\tau,z}$ be any self-adjoint extension of $S_\tau$ in~\eqref{def_S_tau}. Then 
  \begin{equation*}
    \sigma\ti{ess}(A_{\tau,z}) = \bigg(-\infty, -m\left|\frac{\tau^2-4}{\tau^2+4}\right|\bigg] \cup \bigg[m\left|\frac{\tau^2-4}{\tau^2+4}\right|, \infty\bigg).
  \end{equation*}
\end{cor}

\section{Discrete spectrum}\label{section_discrete}

In this section we show the statements about the discrete spectrum of any self-adjoint extension of $S_\tau$ defined in~\eqref{def_S_tau} in Theorem~\ref{theorem_intro}. We begin with the simpler case of a positive interaction strength.  

\begin{prop}
	Let $\tau \in (0,\infty) \setminus \{2\}$, $\w \in (0,\frac{\pi}{2})$,  and let $A_{\tau,z}$ be any self-adjoint extension of $S_\tau$ in~\eqref{def_S_tau}. Then $A_{\tau,z}$ has at most one discrete eigenvalue of multiplicity one.   
\end{prop} 

\begin{proof}
	First, Proposition~\ref{p-vladnote}, Corollary~\ref{corollary_sigma_ess_pos}, and Theorem~\ref{t-quadraticform} imply that $\sigma(S_\tau^*S_\tau) = [m^2,\infty)$. By Corollary~\ref{corollary_sigma_ess_pos} one has $\sigma_{\rm ess}(A_{\tau,z}) = (-\infty,-m]\cup [m,\infty)$.  Recall that by Theorem~\ref{t-nbreakdown} the deficiency indices of $S_\tau$ are $(1,1)$. Hence, by Proposition~\ref{p-vladnote}
	the resolvent difference
	$(S_\tau^*S_\tau+1)^{-1} - \left[(A_{\tau,z})^2+1\right]^{-1}$ is a rank-one operator. The perturbation result~\cite[\S 9.3, Theorem 3]{BS87} implies that the dimension of the spectral subspace of the self-adjoint operator $(A_{\tau,z})^2$ corresponding to the interval $(0,m^2)$ is at most one. As the spectral subspace for $(A_{\tau,z})^2$ corresponding to the interval $(0,m^2)$ coincides with the spectral subspace for $A_{\tau,z}$ corresponding to the interval $(-m,m)$, the result follows.
\end{proof}

\begin{rem}
	As a consequence of the above proposition and of the symmetry of the spectrum of the distinguished self-adjoint extension $A_{\tau,1}$
	stated in Theorem~\ref{t-symmetry},
	we infer that for any $\tau\in(0,\infty)\setminus\{2\}$ the discrete spectrum of $A_{\tau,1}$ in the gap $(-m,m)$ can either be  empty or consist of one simple eigenvalue at $0$.
\end{rem}

The case where $\tau<0$ is more complicated. The proof of the existence of discrete eigenvalues is based on the min-max principle and on the construction of a family of test functions. This is a modification of the approach of~\cite{EN03}, applied there to Schr\"odinger operators with $\delta$-interactions supported on broken lines.
We remark that the angle $\omega_\star$ that is implicitly given in the proposition below is computed explicitly after the proof. Since the explicit value is rather complicated, we prefer to state the implicit value here.

\begin{prop}\label{prop:disc2}
	Let $\tau \in (-\infty,0) \setminus \{-2\}$ and $N\in\N$. Define the number $\omega_\star \in (0, \frac{\pi}{2})$ by 
	\begin{equation*}
	  \omega_\star = \omega_\star(\tau, N) := \max_{L \in (0, \infty)} \arctan \left\{ \frac{N^2 \pi^2 (4+\tau^2)^2 (16 \tau^2 + (4+\tau^2)^2) - 8 m^2 L^2 \tau^2(4-\tau^2)^2}{(2 N^2 \pi^2 + m^2 L^2) [(4-\tau^2)^2 + (4+\tau^2)^2] 4 m \tau L (4+\tau^2)} \right\}.
	\end{equation*}
	If $\omega \in (0,\omega_\star]$, then any self-adjoint extension $A_{\tau,z}$ of $S_\tau$ has at least $N$ discrete eigenvalues with multiplicities taken into account.
\end{prop}

\begin{proof}
Our strategy is to apply the min-max principle to the non-negative operator $(A_{\tau,z})^2$ using a family of test functions lying in the domain of the symmetric operator $S_\tau$. The test functions will be supported in the strip $\Pi_L := [L, 2L]\times\RR$  for some $L>0$. We will also use the notation $d:=L\tan(\w)$. The intersection $\Pi_L\cap\Omega_-$
is the union of two disjoint domains: the subdomain of $\Pi_L\cap\Omega_-$ which lies above $\Gamma_l$ will be denoted  by $\Omega^l_-$ and the subdomain of $\Pi_L\cap\Omega_-$ below $\Gamma_r$ will be denoted by $\Omega^r_-$; see Figure~\ref{f-testfunction}.

\begin{figure}[h]
\centering
\includegraphics[scale=1]{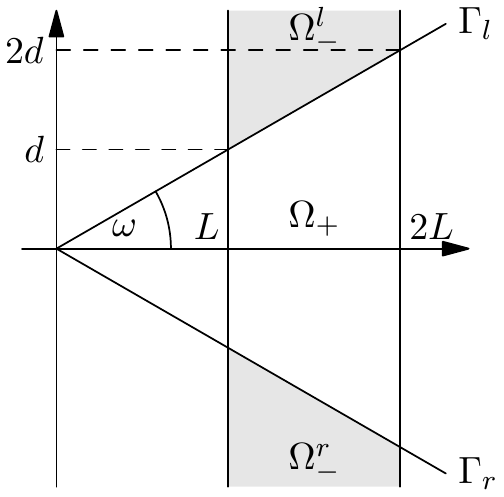}
\caption{Decomposition of $\RR^2$ used for the test functions $u_n(x,y)$. The subdomains $\Omega_-^l$ and $\Omega_-^r$ are depicted in light gray.}
\label{f-testfunction}
\end{figure}

Recall that by Corollary \ref{corollary_sigma_ess_neg} the essential spectra of $(A_{\tau,z})^2$ and of $S_\tau^*S_\tau$ are identified with 
\begin{align}\label{e-lowerbound}
    \sigma_{\rm ess}((A_{\tau,z})^2) = \sigma_{\rm ess}(S_\tau^*S_\tau) = \big[\eps_\tau^2,\infty\big),\qquad
    \text{where}\,\,\, \eps_\tau = m\frac{\tau^2-4}{\tau^2+4}.
\end{align}
Fix $N\in\NN$ and define the following abbreviations:
\begin{align*}
    a:=\frac{4+\tau^2}{4-\tau^2} \hspace{2em}\text{ and }\hspace{2em} b:=\frac{4\tau}{4-\tau^2}.
\end{align*}
We introduce a family of test functions $u_n\colon\mathbb{R}^2\rightarrow\mathbb{C}^2$ of the form $u_n(x,y):=f_n(x) g(y) h(x,y)$, $n\in\{1,2,\dots,N\}$. Here, the functions $f_n\colon \mathbb{R}\rightarrow\mathbb{R}$, $n\in\{1,2,\dots,N\}$, are defined by
\[
    f_n(x)=\sin\left(\frac{2n\pi}{L}x\right)\cdot\chi\ci{[L,2L]}(x), 
\] 
where $\chi_{[L,2L]}$ is the characteristic function of the interval $[L,2L]$.
Furthermore, the function $g :\mathbb{R}\rightarrow\mathbb{R}$ is introduced by    
\[
	g(y)=
    \begin{cases}
    1, & |y|\leq2d, \\
    e^{-\gamma(|y|-2d)}, & |y|>2d,
    \end{cases} 
\]
with $\gamma=-\frac{4m\tau}{4+\tau^2}$.
Finally, the function $h : \mathbb{R}^2\rightarrow\mathbb{C}^2$ is given by
\[
    h(x,y)=
    \begin{cases}
    \begin{pmatrix}
    1 \\
    0
    \end{pmatrix}, & (x,y)\in\Omega_+, \\
    \begin{pmatrix}
    a \\
    be^{i\w}
    \end{pmatrix}, & (x,y)\in\Omega^l_-, \\
    \begin{pmatrix}
    a \\
    -be^{-i\w}
    \end{pmatrix}, & (x,y)\in\Omega^r_-.
    \end{cases}
\]
The vector-valued function $h$ ensures that the transmission condition for $\dom S_\tau$ is satisfied, and so by construction $\{u_n\}_{n=1}^N\subset\dom S_\tau$. Since the functions $f_n$ are linearly independent, also the functions $\{u_n\}_{n=1}^N$ are linearly independent and the subspace $\mathcal{L}_N := {\rm span}\,\{u_1,u_2,\dots,u_N\}$ of $\dom S_\tau$ has dimension $N$. A generic element $u\in\mathcal{L}_N$ is therefore given by $u(x,y)=\sum_{n=1}^N c_nu_n(x,y)$ with some complex coefficients $\{c_n\}_{n=1}^N$. We are going to show that, for $\omega \in (0,\omega_\star]$ and the corresponding $L$ that maximizes the expression in the definition of $\omega_\star$, one has $\| S_\tau u \|_{\mathbb{R}^2}^2 < \varepsilon_\tau^2 \| u \|_{\mathbb{R}^2}^2$ for all $u \in \mathcal{L}_N$.

We will compute all the terms in $\|S_\tau u\|^2_{\R^2}$ from Theorem~\ref{t-quadraticform} for $u\in\mathcal{L}_N$. Begin by noticing that 
\begin{align*}
    \int_{\Gamma_l}\big|u_+|_{\Gamma_l}-u_-|_{\Gamma_l}\big|^2ds=\int_{\Gamma_r}\big|u_+|_{\Gamma_r}-u_-|_{\Gamma_r}\big|^2ds.
\end{align*}
Recall that the edge $\Gamma_l$ is parametrized by $\{(s\cos(\w),s\sin(\w) ):s\in\mathbb{R}_+\}$. We calculate that 
\begin{align*}
    c(\tau):=\Bigg|\begin{pmatrix} a-1 \\ be^{i\w}\end{pmatrix}\Bigg|^2=(a-1)^2+b^2=\frac{4\tau^2(\tau^2+4)}{(\tau^2-4)^2}.
\end{align*}
Hence, we find that
\begin{align*}
    \int_{\Gamma_l}\big|u_+|_{\Gamma_l}-u_-|_{\Gamma_l}\big|^2ds&=\int\ci{L/\cos(\w)}^{2L/\cos(\w)}c(\tau)\Bigg|\sum_{n=1}^Nc_n\sin\left(\frac{2ns\pi}{L}\cos(\w)\right)\Bigg|^2 ds.
\end{align*}
The change of variables $t=(s\pi\cos(\w))/L$ and the orthogonality of the functions $\sin(2 n t)$, $n \in \{ 1, \dots, N\}$, in $L^2(\pi, 2 \pi)$ then yields
\begin{align*}
 \int_{\Gamma_l}\big|u_+|_{\Gamma_l}-u_-|_{\Gamma_l}\big|^2ds&=c(\tau)\frac{L}{\pi\cos(\w)}\int_{\pi}^{2\pi}\Big|\sum_{n=1}^N c_n\sin(2nt)\Big|^2 dt \\
 &=c(\tau)\frac{L}{\pi\cos(\w)}\sum_{n=1}^N\int_{\pi}^{2\pi}|c_n|^2\cdot|\sin(2nt)|^2 dt \\
 &=c(\tau)\frac{L}{\pi\cos(\w)}\sum_{n=1}^N\frac{\pi}{2}|c_n|^2=\frac{c(\tau)L}{2\cos(\w)}\sum_{n=1}^N|c_n|^2.
\end{align*}
This immediately implies that
\begin{align}\label{eq:uGamma}
     \big\|u_+|_\Gamma-u_-|_\Gamma \big\|^2_\Gamma=2\int_{\Gamma_l}\big|u_+|_{\Gamma_l}-u_-|_{\Gamma_l}\big|^2ds=\frac{c(\tau)L}{\cos(\w)}\sum_{n=1}^N|c_n|^2.
\end{align}

For the calculation of $\|u\|^2_{\RR^2}$, we define
\begin{align*}
    \kappa(\tau)=a^2+b^2=\frac{(4+\tau^2)^2+16\tau^2}{(4-\tau^2)^2},
\end{align*}
and notice that
due to the symmetry
of $|u|$ with respect to the $x$-axis one has
\[
	\|u\|^2_{\R^2} = 2\int_{\Pi_L^+}|u|^2dxdy,
\]
where $\Pi_L^+:=[L, 2L]\times[0,\infty)$. We define three disjoint subregions of $\Pi_L^+$ by 
\[
\begin{aligned}
	\Pi_{L,1}^+ &:= \Pi_L^+\cap\Omega_+,\\ 
	\Pi_{L,2}^+& :=\Pi_L^+\cap\{(x,y)\in\Omega^l_-\colon y\le 2d\},\\
	\Pi_{L,3}^+& := \Pi_L^+\cap\{(x,y)\in\Omega^l_-\colon y>2d\}.
\end{aligned}
\] 
Note that $\overline{\Pi_L^+} = \overline{\cup_{j=1}^3\Pi_{L,j}^+}$.
Using the symmetry of $|f_n(x)|$ about $x=3L/2$ we obtain that
\begin{align*}
    \|u\|^2_{\RR^2}&=
    2\int_{\Pi_{L,1}^+}|u|^2dxdy
    +
    2\int_{\Pi_{L,2}^+}|u|^2dxdy   +2\int_{\Pi_{L,3}^+}|u|^2dxdy
    \\
    &=
    2\int_L^{2L}\Big|\sum_{n=1}^N c_n\sin(\frac{2n\pi}{L}x)\Big|^2dx\left\{\frac{3L\tan(\w)}{2}+\frac{L\tan(\w)}{2}\kappa(\tau)+\int_{2d}^{\infty}\kappa(\tau)e^{-2\gamma(y-2d)}dy\right\} \\
    &=2\sum_{n=1}^N|c_n|^2\int_L^{2L}\sin^2\left(\frac{2n\pi}{L}x\right)dx\left\{\frac{L\tan(\w)[3+\kappa(\tau)]}{2}+\frac{\kappa(\tau)}{2\gamma}\right\} \\
    &=\left\{\frac{L^2\tan(\w)[3+\kappa(\tau)]}{2}+\frac{L\kappa(\tau)}{2\gamma}\right\}\sum_{n=1}^N|c_n|^2.
\end{align*}
Finally, we calculate $\|\grad u\|^2_{\RR^2\setminus\Gamma}$
for a generic $u\in\cL_N$. A direct computation yields
\begin{align*}
    \partial_x u(x,y)&=
    \sum_{n=1}^N c_n\frac{2n\pi}{L}\cos\left(\frac{2n\pi}{L}x\right) \chi\ci{[L,2L]}(x)g(y)h(x,y), \\
    \partial_y u(x,y) &=
    -\sum_{n=1}^N c_n\sin\left(\frac{2n \pi}{L}x\right)\chi_{[L, 2L]}(x) \chi\ci{(-\infty,-2d]\cup[2d,\infty)}(y) \sgn(y)\gamma e^{-\gamma(|y|-2d)} h(x,y),
\end{align*}
where it was used in between that $h(x,y)$ is piecewise constant. A calculation similar to that of $\|u\|^2_{\RR^2}$ above then yields 
\begin{align*}
    \|\partial_x u\|^2_{\RR^2\setminus\Gamma}&=2\sum_{n=1}^N |c_n|^2 \left(\frac{2n\pi}{L}\right)^2\int_L^{2L}\cos^2\left(\frac{2n\pi}{L}x\right)dx\left\{\frac{L\tan(\w)[3+\kappa(\tau)]}{2}+\frac{\kappa(\tau)}{2\gamma}\right\} \\
    &=\sum_{n=1}^N |c_n|^2 \frac{(2n\pi)^2}{L}\left\{\frac{L\tan(\w)[3+\kappa(\tau)]}{2}+\frac{\kappa(\tau)}{2\gamma}\right\}.
\end{align*}
Furthermore, we obtain that
\begin{align*}
    \|\partial_y u\|^2_{\RR^2\setminus\Gamma}&=\sum_{n=1}^N|c_n|^2\int_L^{2L}\sin^2\left(\frac{2n\pi}{L}x\right)dx\cdot 2\int_{2d}^{\infty}\gamma^2\kappa(\tau)e^{-2\gamma(y-2d)}dy \\
    &=\sum_{n=1}^N|c_n|^2\cdot \frac{L}{2}\cdot\gamma\kappa(\tau).
\end{align*}
After subtracting the bottom of the essential spectrum, we finally can calculate that
\begin{align*}
    \|S_\tau u\|^2_{\RR^2}-\eps_\tau^2\|u\|^2_{\RR^2}&=\|\grad u\|^2_{\RR^2\setminus\Gamma}+m^2\left[1-\left(\frac{\tau^2-4}{\tau^2+4}\right)^2\right]\|u\|^2_{\RR^2}+\frac{2m}{\tau}\big\|u_+|_\Gamma-u_-|_\Gamma \big\|^2_{\Gamma} \\
    &=\sum_{n=1}^N|c_n|^2\Bigg\{\frac{(2n\pi)^2}{L}\left\{\frac{L\tan(\w)[3+\kappa(\tau)]}{2}+\frac{\kappa(\tau)}{2\gamma}\right\}+\frac{L\gamma\kappa(\tau)}{2} \\
    &\qquad+\frac{16m^2\tau^2}{(\tau^2+4)^2}\left\{\frac{L^2\tan(\w)[3+\kappa(\tau)]}{2}+\frac{L\kappa(\tau)}{2\gamma}\right\}+\frac{2mLc(\tau)}{\tau\cos(\w)}\Bigg\}.
\end{align*}
Next, by using that for $\omega < \frac{\pi}{2}$ one has $\cos \omega < 1$, $\frac{8 \tau^2}{(\tau^2+4)^2} \leq 1$, and
\begin{align*}
    \frac{L\kappa(\tau)}{2\gamma}&\left[\frac{16m^2\tau^2}{(\tau^2+4)^2}+\gamma^2\right]+\frac{2mLc(\tau)}{\tau}=\frac{L\kappa(\tau)}{2\gamma}\cdot 2\gamma^2+\frac{2mLc(\tau)}{\tau} \\
    &=\frac{4mL\tau}{(4-\tau^2)^2(4+\tau^2)}\left[-16\tau^2+(4+\tau^2)^2\right]=\frac{4mL\tau}{(4-\tau^2)^2(4+\tau^2)}(4-\tau^2)^2 = \frac{4mL\tau}{4+\tau^2},
\end{align*}
we can estimate for $u\in\cL_N\setminus\{0\}$
\begin{align} \label{equation_estimate_form}
    \|S_\tau u\|^2_{\RR^2}-\eps_\tau^2\|u\|^2_{\RR^2}
    <\sum_{n=1}^N|c_n|^2 \left\{ \tan(\w)[3+\kappa(\tau)] ( 2 N^2 \pi^2 + m^2 L^2 ) + \frac{4 m L \tau}{4 + \tau^2} + \frac{2 N^2\pi^2 \kappa(\tau)}{L \gamma} \right\}.
\end{align}
Hence, if we choose for $L>0$
\begin{equation*}
  \begin{split}
    \omega &= \omega(L) := \arctan \left\{ \frac{-4 m L^2 \tau \gamma - 2 N^2 \pi^2 \kappa(\tau) (4+\tau^2)}{(2 N^2 \pi^2 + m^2 L^2) [3 + \kappa(\tau)] L \gamma (4+\tau^2)} \right\} \\
    &= \arctan \left\{ \frac{N^2 \pi^2 (4+\tau^2)^2 (16 \tau^2 + (4+\tau^2)^2) - 8 m^2 L^2 \tau^2(4-\tau^2)^2}{(2 N^2 \pi^2 + m^2 L^2) [(4-\tau^2)^2 + (4+\tau^2)^2] 4 m \tau L (4+\tau^2)} \right\},
  \end{split}
\end{equation*}
then the right hand side in~\eqref{equation_estimate_form} is zero, implying that $\|S_\tau u\|^2_{\RR^2}-\eps_\tau^2\|u\|^2_{\RR^2}<0$ for all $u \in \mathcal{L}_N\setminus\{0\}$. Note that the function $\omega(L)$ has a unique positive maximum for $L>0$ and hence, this claim remains true for the maximum $\omega_\star := \max_{L>0} \omega(L)$.

Let $A_{\tau,z}$ be any self-adjoint extension of $S_\tau$. The inclusion $\cL_N \subset\dom A_{\tau,z}$ then holds, as $\dom S_\tau \subset \dom A_{\tau,z}$. Let $L_\star >0$ denote the argument for which the maximum for $\omega_\star$ is attained. Then, for any $u\in \cL_N\setminus\{0\}$ and $\w\in(0,\w_\star]$, we get
\[
	\|A_{\tau,z} u\|^2_{\RR^2} - \eps_\tau^2\|u\|_{\RR^2}^2 = 
	\|S_\tau u\|^2_{\RR^2} - \eps_\tau^2\|u\|_{\RR^2}^2 < 0. 
\]
The characterization of the essential spectrum in~\eqref{e-lowerbound} and the min-max principle then imply that the self-adjoint operator $(A_{\tau,z})^2$ has at least $N$ discrete eigenvalues, with multiplicities taken into account. Thus, we conclude by Corollary~\ref{corollary_sigma_ess_neg} that $A_{\tau,z}$ has the same lower bound on the number of discrete eigenvalues with multiplicities taken into account. 
\end{proof}

The maximum in Proposition~\ref{prop:disc2} can be found explicitly. Let us introduce the functions 
\[
\begin{aligned}
F(\tau) &:=  N^2\pi^2(4+\tau^2)^2(16\tau^2+(4+\tau^2)^2),\qquad
G(\tau) := [(4-\tau^2)^2+(4+\tau^2)^2]4|\tau|(4+\tau^2),\\
H(\tau) &:= 8\tau^2(4-\tau^2)^2.
\end{aligned}\]
Clearly, the maximum is attained for $L>0$ satisfying the condition $m^2 L^2H(\tau) > F(\tau)$. Let us introduce a new variable $x > 0$ through the relation $x = m^2L^2H(\tau) - F(\tau)$. 
Then we have
\[
	\omega_\star = \max_{x\in(0,\infty)}\arctan\left\{\frac{x H(\tau)^{3/2}}{G(\tau)\big(2N^2\pi^2 H(\tau) +F(\tau)+x\big)\sqrt{F(\tau)+x}}\right\}.
\]
Thus, we need to find the point $x >0$ that maximizes the function
\[
	J(x) := \frac{x}{\big(2N^2\pi^2H(\tau) +F(\tau)+x\big)\sqrt{F(\tau)+x}}.
\]
Differentiating with respect to $x$ we find
\[
\begin{aligned}
J'(x) = 
\frac{-x^2 +x\big[2N^2\pi^2H(\tau)+F(\tau)\big] +2F(\tau)\big[2N^2
\pi^2H(\tau)+F(\tau)\big]}{2\big[2N^2\pi^2H(\tau)+F(\tau)+x\big]^2(F(\tau)+x)^{3/2}}.
\end{aligned}
\]
The maximum of $J(x)$ is located at the unique positive solution of the equation $J'(x) = 0$, which is given by
\[
	x_\star(\tau) = 
	N^2\pi^2H(\tau)+ \frac{F(\tau)}{2}+\sqrt{\left(N^2\pi^2H(\tau)+\frac{F(\tau)}{2}\right)^2+4F(\tau)\left(N^2\pi^2H(\tau)+\frac{F(\tau)}{2}\right)}.
\] 
Hence, we arrive at
\[	
	\omega_\star = \arctan\left\{
	\frac{x_\star(\tau)H(\tau)^{3/2}}{G(\tau)\big(2N^2\pi^2H(\tau)+F(\tau)+x_\star(\tau)\big)\sqrt{F(\tau)+
			x_\star(\tau)}}
	\right\}.
\]
From this formula, one gets that $\omega_\star\rightarrow 0$ in the limits $\tau\rightarrow 0^-$, $\tau\rightarrow -2$, and $\tau\rightarrow-\infty$ for any fixed $N\in\NN$.

\begin{appendix}
\section{On a Sobolev space with a piecewise constant weight}

Recall that $M$ is defined by~\eqref{def_M} and that $\partial_t$ is given by~\eqref{def_partial_t}. In this appendix we prove for any $u \in \dom S_\tau$ the identity
\begin{equation}\label{equation_commutation}
\begin{aligned}
    \big\langle M u_+|_\Gamma,&i\sigma_3(\partial_t (Mu_+|_{\Gamma})\big\rangle_{H^{1/2}(\Gamma; \mathbb{C}^2)\times H^{-1/2}(\Gamma; \mathbb{C}^2)}\\
    &= \big\langle M^{-1} M u_+|_\Gamma, i\sigma_3 \partial_t (u_+|_{\Gamma})\big\rangle_{H^{1/2}(\Gamma; \mathbb{C}^2)\times H^{-1/2}(\Gamma; \mathbb{C}^2)},
\end{aligned}
\end{equation}
which appears in the proof of Theorem ~\ref{t-quadraticform} in equation \eqref{e-boundary0s}. Note that $u_+|_\Gamma \in H^{1/2}(\Gamma; \mathbb{C}^2)$ and $M u_+|_\Gamma = u_-|_\Gamma \in H^{1/2}(\Gamma; \mathbb{C}^2)$. Let $\gamma\colon \mathbb{R} \rightarrow \mathbb{R}^2$ be an arc length parametrization of $\Gamma$ such that $\gamma(\mathbb{R}_-) = \Gamma_r$ and $\gamma(\mathbb{R}_+) = \Gamma_l$, set $\wt{\nu} := \nu \circ \gamma$, and define for $s \in [0,1]$ the weighted Sobolev space
\begin{equation} \label{def_H_nu}
  H^s_{\wt{\nu}}(\mathbb{R}; \mathbb{C}^2) := \big\{ f \in H^s(\mathbb{R}; \mathbb{C}^2): i \sigma_3 (\sigma \cdot \wt{\nu}) f \in H^s(\mathbb{R}; \mathbb{C}^2) \big\},
\end{equation}
which is a Hilbert space when endowed with the norm
\begin{equation*}
  \| f \|_{H^s_{\wt{\nu}}(\mathbb{R}; \mathbb{C}^2)}^2:= \| f \|_{H^s(\mathbb{R}; \mathbb{C}^2)}^2 + \| i \sigma_3 (\sigma \cdot \wt{\nu}) f \|_{H^s(\mathbb{R}; \mathbb{C}^2)}^2, \qquad f \in H^s_{\wt{\nu}}(\mathbb{R}; \mathbb{C}^2).
\end{equation*}
The definition of $M$ in \eqref{def_M}  then implies that $(u_+|_\Gamma) \circ \gamma \in H^{1/2}_{\wt{\nu}}(\mathbb{R}; \mathbb{C}^2)$. Hence, we conclude that~\eqref{equation_commutation} is true if we can show the following proposition:

\begin{prop} \label{proposition_commutation_H_nu}
  Let $f, g \in H^{1/2}_{\widetilde{\nu}}(\mathbb{R}; \mathbb{C}^2)$. Then,
  \begin{equation*}
    \begin{split}
      \big\langle (i \sigma_3 (\sigma \cdot \wt{\nu}) f)',  g&\big\rangle_{H^{-1/2}(\mathbb{R}; \mathbb{C}^2)\times H^{1/2}(\mathbb{R}; \mathbb{C}^2)}\\
      &= \big\langle f', i \sigma_3 (\sigma \cdot \wt{\nu}) g \big\rangle_{H^{-1/2}(\mathbb{R}; \mathbb{C}^2)\times H^{1/2}(\mathbb{R}; \mathbb{C}^2)}
    \end{split}
  \end{equation*}
  holds.
\end{prop}

For the proof of Proposition~\ref{proposition_commutation_H_nu}, some preliminary considerations on the spaces $H^s_{\widetilde{\nu}}(\mathbb{R}; \mathbb{C}^2)$ are necessary. First, we consider the cases $s=0$ and $s=1$. In the following, we denote by $[X, Y]_\theta$, $\theta \in [0,1]$, the interpolation space between two Hilbert spaces $X \subset Y$ with continuous and dense embedding; see \cite[Chapter~I, \S 15]{LM} and also~\cite[Appendix B]{M}. 
\begin{lem} \label{lemma_H_nu1}
  Let the space $H^s_{\wt{\nu}}(\mathbb{R}; \mathbb{C}^2)$, $s \in [0,1]$, be defined as in~\eqref{def_H_nu}. Then, the following holds:
  \begin{itemize}
    \item[(i)] $H^0_{\wt{\nu}}(\mathbb{R}; \mathbb{C}^2) = L^2(\mathbb{R}; \mathbb{C}^2)$, with equivalent norms, and multiplication by $i \sigma_3 (\sigma \cdot \wt{\nu})$ defines a bounded map in $L^2(\mathbb{R}; \mathbb{C}^2)$.
    \item[(ii)] $H^1_{\wt{\nu}}(\mathbb{R}; \mathbb{C}^2) = H^1_0(\mathbb{R} \setminus \{ 0 \}; \mathbb{C}^2)$, with equivalent norms, one has 
    \begin{equation} \label{commutation_derivative_H^1_0}
      \big( i \sigma_3 (\sigma \cdot \wt{\nu}) f \big)' = i \sigma_3 (\sigma \cdot \wt{\nu}) f'
    \end{equation}
    for all $f \in H^1_{\wt{\nu}}(\mathbb{R}; \mathbb{C}^2)$, and multiplication by $i \sigma_3 (\sigma \cdot \wt{\nu})$ defines a bounded map in $H^1_0(\mathbb{R}\setminus \{ 0 \}; \mathbb{C}^2)$.
  \end{itemize}
  In particular, for any $\theta \in [0,1]$ the map
  \begin{equation} \label{equation_M_bounded}
    i \sigma_3 (\sigma \cdot \wt{\nu})\colon \big[ H^1_0(\mathbb{R} \setminus \{ 0 \}; \mathbb{C}^2), L^2(\mathbb{R}; \mathbb{C}^2) \big]_\theta \rightarrow \big[ H^1_0(\mathbb{R} \setminus \{ 0 \}; \mathbb{C}^2), L^2(\mathbb{R}; \mathbb{C}^2) \big]_\theta
  \end{equation}
  is bounded.
\end{lem}
\begin{proof}
  (i) Since $|\wt{\nu}|=1$, the anti-commutation relations in~\eqref{Pauli_anti_commutation} imply that $i \sigma_3 (\sigma \cdot \wt{\nu})$ is pointwise unitary. Hence, the claims in item~(i) are clear.
  
  (ii) Let $f \in H^1_{\wt{\nu}}(\mathbb{R}; \mathbb{C}^2)$. Then,  by definition $f \in H^1(\mathbb{R}; \mathbb{C}^2)$ and $i \sigma_3 (\sigma \cdot \wt{\nu}) f \in H^1(\mathbb{R}; \mathbb{C}^2)$. This implies for any $\varphi \in C_0^\infty(\mathbb{R}; \mathbb{C}^2)$ that
  \begin{equation*}
    \begin{split}
      \int_\mathbb{R} &\big( i \sigma_3 (\sigma \cdot \wt{\nu}) f \big)' \varphi dx
      = -\int_\mathbb{R} i \sigma_3 (\sigma \cdot \wt{\nu}) f \varphi' dx \\
      &= \int_\mathbb{R_-} i \sigma_3 (\sigma \cdot \wt{\nu}_r) f' \varphi dx + \int_\mathbb{R_+} i \sigma_3 (\sigma \cdot \wt{\nu}_l) f' \varphi dx - i \sigma_3 (\sigma \cdot (\wt{\nu}_r - \wt{\nu}_l)) f(0) \varphi(0),
    \end{split}
  \end{equation*}
  where the notation $\wt{\nu}_{l/r} = \wt{\nu}|_{\mathbb{R}_{+/-}}$ and integration by parts in $\mathbb{R}_\pm$ were used in the last step. Since $\wt{\nu}_l \neq \wt{\nu}_r$, the anti-commutation relations in~\eqref{Pauli_anti_commutation} imply that $i \sigma_3 (\sigma \cdot (\wt{\nu}_r - \wt{\nu}_l))$ is invertible and hence $f(0)$ must be $0$, i.e. $f \in H^1_0(\mathbb{R} \setminus \{ 0 \}; \mathbb{C}^2)$. Moreover, the above calculation shows~\eqref{commutation_derivative_H^1_0} for $f \in H^1_{\wt{\nu}}(\mathbb{R}; \mathbb{C}^2)$. 
  
  On the other hand, a similar consideration as above yields for $f \in H^1_0(\mathbb{R} \setminus \{ 0 \}; \mathbb{C}^2)$ that $f \in H^1_{\wt{\nu}}(\mathbb{R}; \mathbb{C}^2)$ and that~\eqref{commutation_derivative_H^1_0} holds in this case as well. Eventually, since $i \sigma_3 (\sigma \cdot \wt{\nu})$ is pointwise unitary, it follows from~\eqref{commutation_derivative_H^1_0} that the norms in $H^1_{\wt{\nu}}(\mathbb{R}; \mathbb{C}^2)$ and $H^1_0(\mathbb{R} \setminus \{ 0 \}; \mathbb{C}^2)$ are equivalent. 
  
  Finally, as the multiplication by $i \sigma_3 (\sigma \cdot \wt{\nu})$ gives rise to a bounded map in both spaces $H^1_0(\mathbb{R} \setminus \{ 0 \}; \mathbb{C}^2)$ and  $L^2(\mathbb{R}; \mathbb{C}^2)$ by items~(ii) and~(i), respectively, it follows by interpolation that the map in~\eqref{equation_M_bounded} is well defined and bounded.
\end{proof}

In the next lemma we characterize $H^{1/2}_{\wt{\nu}}(\mathbb{R}; \mathbb{C}^2)$. For this purpose, recall that the space $H^{1/2}_{00}(\mathbb{R} \setminus \{ 0 \}; \mathbb{C}^2)$ is defined by
\begin{equation*}
  H^{1/2}_{00}(\mathbb{R} \setminus \{ 0 \}; \mathbb{C}^2) = \left\{ f \in H^{1/2}(\mathbb{R}; \mathbb{C}^2): \int_{\mathbb{R}} \frac{|f(x)|^2}{|x|} dx < \infty \right\}
\end{equation*}
and that this space, endowed with the norm
\begin{equation*}
  \| f \|_{H^{1/2}_{00}(\mathbb{R} \setminus \{ 0 \}; \mathbb{C}^2)}^2 := \| f \|_{H^{1/2}(\mathbb{R}; \mathbb{C}^2)}^2 + \int_{\mathbb{R}} \frac{|f(x)|^2}{|x|} dx, \qquad f \in H^{1/2}_{00}(\mathbb{R} \setminus \{ 0 \}; \mathbb{C}^2),
\end{equation*}
is a Hilbert space. We are going to use 
\begin{equation} \label{interpolation_H_00}
  H^{1/2}_{00}(\mathbb{R} \setminus \{ 0 \}; \mathbb{C}^2) = \big[ H^1_0(\mathbb{R} \setminus \{ 0 \}; \mathbb{C}^2), L^2(\mathbb{R}; \mathbb{C}^2) \big]_{1/2},
\end{equation}
cf. \cite[Chapter~I, Theorem~11.7 and Remark~11.5]{LM}. In particular, by elementary properties of interpolation spaces this implies that $H^1_0(\mathbb{R} \setminus \{ 0 \}; \mathbb{C}^2)$ is dense in $H^{1/2}_{00}(\mathbb{R} \setminus \{ 0 \}; \mathbb{C}^2)$.

\begin{lem} \label{lemma_H_nu2}
  Let $H^{1/2}_{\wt{\nu}}(\mathbb{R}; \mathbb{C}^2)$ be defined by~\eqref{def_H_nu}. Then, $H^{1/2}_{\wt{\nu}}(\mathbb{R}; \mathbb{C}^2) = H^{1/2}_{00}(\mathbb{R} \setminus \{ 0 \}; \mathbb{C}^2)$.
\end{lem} 
\begin{proof}
  First, it follows from~\eqref{equation_M_bounded} and~\eqref{interpolation_H_00} that $H^{1/2}_{00}(\mathbb{R} \setminus \{ 0 \}; \mathbb{C}^2) \subset H^{1/2}_{\wt{\nu}}(\mathbb{R}; \mathbb{C}^2)$. To show the converse inclusion, let $f \in H^{1/2}_{\wt{\nu}}(\mathbb{R}; \mathbb{C}^2)$. Set $\wt{\nu}_r = \wt{\nu}(-1)$ and $\wt{\nu}_l = \wt{\nu}(1)$. Then, as $\wt{\nu}_r$ is constant, $f \in H^{1/2}(\mathbb{R}; \mathbb{C}^2)$ implies that $i \sigma_3 (\sigma \cdot \wt{\nu}_r) f \in H^{1/2}(\mathbb{R}; \mathbb{C}^2)$ and thus also
  \begin{equation*}
    i \sigma_3 (\sigma \cdot (\wt{\nu} - \wt{\nu}_r)) f =  i \sigma_3 (\sigma \cdot (\wt{\nu}_l - \wt{\nu}_r)) \chi_{\mathbb{R}_+} f \in H^{1/2}(\mathbb{R}; \mathbb{C}^2),
  \end{equation*}
  where $\chi_{\mathbb{R}_+}$ denotes the characteristic function for $\mathbb{R}_+$. Since the anti-commutation relations in~\eqref{Pauli_anti_commutation} imply that $ i \sigma_3 (\sigma \cdot (\wt{\nu}_l - \wt{\nu}_r))$ is invertible, we conclude that $\chi_{\mathbb{R}_+} f \in H^{1/2}(\mathbb{R}; \mathbb{C}^2)$. Making use of the equivalent Sobolev-Slobodeckii norm in $H^{1/2}(\mathbb{R}; \mathbb{C}^2)$, this implies that
  \begin{equation*}
   \begin{split}
     \infty &> \int_\mathbb{R} \int_\mathbb{R} \frac{|(\chi_{\mathbb{R}_+} f)(x) - (\chi_{\mathbb{R}_+} f)(y)|^2}{(y-x)^2} dy dx 
     \geq \int_0^\infty \int_{-\infty}^0 \frac{|f(x)|^2}{(y-x)^2} dy dx 
     = \int_0^\infty \frac{|f(x)|^2}{x} dx.
   \end{split}
  \end{equation*}
  In a similar way as above, the relations $i \sigma_3 (\sigma \cdot \wt{\nu}_l) f \in H^{1/2}(\mathbb{R}; \mathbb{C}^2)$ and $i \sigma_3 (\sigma \cdot \wt{\nu}) f \in H^{1/2}(\mathbb{R}; \mathbb{C}^2)$ imply that $\chi_{\mathbb{R}_-} f \in H^{1/2}(\mathbb{R}; \mathbb{C}^2)$ and 
  \begin{equation*}
   \begin{split}
      -\int_{-\infty}^0 \frac{|f(x)|^2}{x} dx < \infty.
   \end{split}
  \end{equation*}
  Hence, $f \in H^{1/2}_{00}(\mathbb{R} \setminus \{ 0 \}; \mathbb{C}^2)$, which finishes the proof of this lemma.
\end{proof}

Now, we are prepared to prove Proposition~\ref{proposition_commutation_H_nu}:

\begin{proof}[Proof of Proposition~\ref{proposition_commutation_H_nu}]
  Let $f, g \in H^{1/2}_{\wt{\nu}}(\mathbb{R}; \mathbb{C}^2)$. Since $H^1_0(\mathbb{R} \setminus \{ 0 \}; \mathbb{C}^2)$ is dense in the space $H^{1/2}_{00}(\mathbb{R} \setminus \{ 0 \}; \mathbb{C}^2) = H^{1/2}_{\wt{\nu}}(\mathbb{R}; \mathbb{C}^2)$, there exists a sequence $\{f_n\}_{n\in\NN} \subset H^1_0(\mathbb{R} \setminus \{ 0 \}; \mathbb{C}^2)$ such that 
  \begin{equation*}
    \| f_n - f \|_{H^{1/2}_{00}(\mathbb{R} \setminus \{ 0 \}; \mathbb{C}^2)} \rightarrow 0, \quad \text{as} \quad n \rightarrow \infty.
  \end{equation*}
  Taking the definition of the norm in $H^{1/2}_{00}(\mathbb{R} \setminus \{ 0 \}; \mathbb{C}^2)$ into account this implies, in particular, that $f_n \rightarrow f$ in $H^{1/2}(\mathbb{R}; \mathbb{C}^2)$ and thus, also $f_n' \rightarrow f'$ in $H^{-1/2}(\mathbb{R}; \mathbb{C}^2)$.
  
  Next, as the multiplication by $i \sigma_3 (\sigma \cdot \wt{\nu})$ gives rise to a bounded operator in $H^{1/2}_{00}(\mathbb{R} \setminus \{ 0 \}; \mathbb{C}^2)$ by~\eqref{equation_M_bounded} and~\eqref{interpolation_H_00}, we get that 
  \begin{equation*}
    \big\| i \sigma_3 (\sigma \cdot \wt{\nu}) (f_n-f) \big\|_{H^{1/2}(\mathbb{R}; \mathbb{C}^2)} \leq \big\| i \sigma_3 (\sigma \cdot \wt{\nu}) (f_n-f) \big\|_{H^{1/2}_{00}(\mathbb{R} \setminus \{ 0 \}; \mathbb{C}^2)} \rightarrow 0, \quad \text{as} \quad n \rightarrow \infty.
  \end{equation*}
  In particular, $i \sigma_3 (\sigma \cdot \wt{\nu}) f_n \rightarrow i \sigma_3 (\sigma \cdot \wt{\nu}) f$ in $H^{1/2}(\mathbb{R}; \mathbb{C}^2)$ and $(i \sigma_3 (\sigma \cdot \wt{\nu}) f_n)' \rightarrow (i \sigma_3 (\sigma \cdot \wt{\nu}) f)'$ in $H^{-1/2}(\mathbb{R}; \mathbb{C}^2)$.
  Hence, using~\eqref{commutation_derivative_H^1_0}, we conclude
  \begin{equation*}
    \begin{split}
      \big\langle (i \sigma_3 (\sigma \cdot \wt{\nu}) f)', g&\big\rangle_{H^{-1/2}(\mathbb{R}; \mathbb{C}^2)\times H^{1/2}(\mathbb{R}; \mathbb{C}^2)}\\
      &= \lim_{n \rightarrow \infty} \big\langle (i \sigma_3 (\sigma \cdot \wt{\nu}) f_n)', g\big\rangle_{H^{-1/2}(\mathbb{R}; \mathbb{C}^2)\times H^{1/2}(\mathbb{R}; \mathbb{C}^2)} \\
      &= \lim_{n \rightarrow \infty} \int_\mathbb{R} i \sigma_3 (\sigma \cdot \wt{\nu}) f_n' \overline{g} dx
      = \lim_{n \rightarrow \infty} \int_\mathbb{R}  f_n' \overline{i \sigma_3 (\sigma \cdot \wt{\nu}) g} dx\\
      &= \lim_{n \rightarrow \infty} \big\langle f_n', i \sigma_3 (\sigma \cdot \wt{\nu}) g \big\rangle_{H^{-1/2}(\mathbb{R}; \mathbb{C}^2)\times H^{1/2}(\mathbb{R}; \mathbb{C}^2)}\\
      &= \big\langle f', i \sigma_3 (\sigma \cdot \wt{\nu}) g \big\rangle_{H^{-1/2}(\mathbb{R}; \mathbb{C}^2)\times H^{1/2}(\mathbb{R}; \mathbb{C}^2)},
    \end{split}
  \end{equation*}
  which is the claimed identity.
\end{proof}
\begin{rem}
    The property $u_\pm|_\Gamma \in H^{1/2}(\RR;\CC^2)$ for any $u\in\dom S_\tau$ immediately follows from the trace theorem. However, the construction in this appendix obtains the stronger statement that $(u_\pm|_{\Gamma}\circ \gamma) \in H^{1/2}_{00}(\RR\setminus\{0\};\CC^2)$ for any $u\in\dom S_\tau$ as a by-product. As before we identify $\Gamma$ with $\RR$ here in such a way that the origin of the real axis coincides with the corner of the broken line.
\end{rem}

\section{Bound for a one-dimensional auxiliary problem} \label{appendix_auxiliary_problem}

Define for $\gamma > 0$ and $\tau\in (-\infty,0)\setminus\{-2\}$ in $L^2((-\gamma, \gamma); \mathbb{C}^2)$ the semibounded and closed  quadratic form $\mathfrak{d}_\gamma$ by
\begin{equation*}
  \begin{split}
    \mathfrak{d}_\gamma[f] &= \int_{(-\gamma, \gamma) \setminus \{ 0 \}} |f'(x)|^2 dx + m^2 \int_{(-\gamma, \gamma)} |f(x)|^2 dx 
     +\frac{2m}{\tau} |f(0^+)-f(0^-)|^2, \\
     \dom \mathfrak{d}_\gamma &= \big\{ u \in H^1((-\gamma, \gamma) \setminus \{ 0 \}; \mathbb{C}^2)\colon f(0^-)=\wt{M} f(0^+) \big\},
  \end{split}
\end{equation*}
where $\wt{M}$ is given by~\eqref{def_wt_M}.
This appendix is devoted to the proof that for all $f \in \dom \mathfrak{d}_\gamma$ one has the bound
\begin{equation} \label{bound_1D_auxiliary}
  \begin{split}
   \mathfrak{d}_\gamma[f] \geq E(\gamma) \int_{(-\gamma, \gamma)} |f(x)|^2 dx 
  \end{split}
\end{equation}
with $E(\gamma) \in [0, m^2)$ such that
\begin{equation} \label{estimate_E_gamma}
    \lim_{\gamma \rightarrow  \infty} E(\gamma) = m^2\left(\frac{\tau^2-4}{\tau^2+4}\right)^2,
\end{equation}
which is used in~\eqref{estimate_lower_bound} in the proof of Proposition~\ref{p-essspec1}. To verify~\eqref{bound_1D_auxiliary} we follow closely the proof of the very similar result \cite[Lemma~4.13]{HOBP}, but since the situation in \cite{HOBP} is slightly different from ours and for completeness we give a full proof here.

The strategy to show~\eqref{bound_1D_auxiliary}--\eqref{estimate_E_gamma} is to determine the self-adjoint operator $D_\gamma$ that is associated with $\mathfrak{d}_\gamma$ via Kato's first representation theorem and to estimate its spectrum. As $\dom D_\gamma \subset \dom \mathfrak{d}_\gamma$ is compactly embedded in $L^2((-\gamma, \gamma); \mathbb{C}^2)$, its spectrum is purely discrete and hence, it suffices to estimate the smallest eigenvalue $E(\gamma)$ of $D_\gamma$. Since one has $C_0^\infty((-\gamma, \gamma) \setminus \{ 0 \}; \mathbb{C}^2) \subset \dom \mathfrak{d}_\gamma$, one finds that $\dom D_\gamma \subset H^2((-\gamma, \gamma) \setminus \{ 0 \}; \mathbb{C}^2) \cap \dom \mathfrak{d}_\gamma$ and that the action of $D_\gamma$ is $D_\gamma f = -f'' + m^2 f$ in $(-\gamma, \gamma) \setminus \{ 0 \}$ for $f \in \dom D_\gamma$. Moreover, using integration by parts, one gets for any $f \in \dom D_\gamma$ and any $g \in \dom \mathfrak{d}_\gamma$
\begin{equation*}
  \begin{split}
    \mathfrak{d}_\gamma[f,g] &= \int_{(-\gamma, \gamma) \setminus \{ 0 \}} (-f'' + m^2 f) \overline{g} dx + f'(\gamma) \overline{g(\gamma)} - f'(0^+) \overline{g(0^+)} \\
    &\qquad + f'(0^-) \overline{g(0^-)} - f'(-\gamma) \overline{g(-\gamma)} + \frac{2 m}{\tau} (f(0^+)-f(0^-)) \overline{(g(0^+)-g(0^-))} \\
    &= \int_{(-\gamma, \gamma) \setminus \{ 0 \}} D_\gamma f \overline{g} dx + f'(\gamma) \overline{g(\gamma)} - f'(-\gamma) \overline{g(-\gamma)} \\
    &\qquad + \left[ \wt{M} f'(0^-) - f'(0^+)  + \frac{2 m}{\tau} (\sigma_0-\wt{M})^2 f(0^+) \right] \overline{g(0^+)},
  \end{split}
\end{equation*}
where in the second step the transmission condition for $f, g \in \dom \mathfrak{d}_\gamma$ was taken into account. Thus, we conclude that
\begin{equation*}
  \begin{split}
    \dom D_\gamma = \bigg\{ f \in H^2((-\gamma, \gamma) \setminus \{ 0 \}; \mathbb{C}^2)&
    \colon f(0^-)=\wt{M} f(0^+), f'(\gamma) = f'(-\gamma) =0,  \\
    &\qquad  \wt{M} f'(0^-) - f'(0^+)  + \frac{2 m}{\tau} (\sigma_0-\wt{M})^2 f(0^+) = 0 \bigg\}.
  \end{split}
\end{equation*}
To compute the smallest eigenvalue $E = E(\gamma)$ of $D_\gamma$, we set $k = \sqrt{m^2 - E}$ and note first that a corresponding eigenfunction $f_E$ must be of the form 
\begin{equation*}
  f_E(x) = \begin{cases} a_+ e^{-k x} + b_+ e^{kx},& x \in(0, \gamma), \\ a_- e^{k x} + b_- e^{-kx},& x \in(-\gamma, 0), \end{cases}
\end{equation*}
for some $a_\pm, b_\pm \in \mathbb{C}^2$ and it has to fulfil the conditions for $f \in \dom D_\gamma$. From the condition $f_E'(\pm \gamma) = 0$ we conclude that $a_\pm = b_\pm e^{2 k \gamma}$. Taking this into account, one finds that $f_E(0^-) = \wt{M} f_E(0^+)$ yields $b_- = \wt{M} b_+$. Hence, the last condition for $f_E \in \dom D_\gamma$ simplifies to
\begin{equation} \label{condition_A_gamma}
  \begin{split}
    0 &= \wt{M} f'_E(0^-) - f'_E(0^+)  + \frac{2 m}{\tau} (\sigma_0-\wt{M})^2 f_E(0^+) \\
    &= k (e^{2 k \gamma} - 1) (\wt{M}^2 + \sigma_0) b_+ + \frac{2 m}{\tau} (e^{2 k \gamma} + 1) (\sigma_0-\wt{M})^2 b_+.
  \end{split}
\end{equation}
By the definition of $\wt{M}$ one finds with a direct calculation that $\wt{M}^2 + \sigma_0 = 2 \frac{4+\tau^2}{4-\tau^2} \wt{M}$ and $(\sigma_0-\wt{M})^2 = \frac{4 \tau^2}{4-\tau^2} \wt{M}$. Since $\wt{M}$ is invertible, we find that~\eqref{condition_A_gamma} is equivalent to
\begin{equation} \label{equation_ev_A_gamma}
  -\frac{4 m \tau}{4+\tau^2} = k \frac{e^{k \gamma} - e^{-k \gamma}}{e^{k \gamma} + e^{-k \gamma}} = k \tanh(k \gamma) =: F_\gamma(k).
\end{equation}
Since 
\begin{equation*}
  F_\gamma'(k) = \tanh(k \gamma) + \frac{k \gamma}{\cosh^2(k \gamma)} > 0 \quad \text{for} \quad k > 0,
\end{equation*}
$\lim_{k \rightarrow 0^+}F(k\gamma)=0$, and $\lim_{k \rightarrow \infty}F(k\gamma)=\infty$, we find that $F_\gamma: (0, \infty) \rightarrow (0, \infty)$ is bijective. Hence, the equation~\eqref{equation_ev_A_gamma} has exactly one solution $k_\gamma>0$, which is associated with the smallest eigenvalue of $D_\gamma$; we remark that~\eqref{equation_ev_A_gamma} has infinitely many solutions $k$ of the form $k=i x$ with $x \in \mathbb{R}$, which are, due to $k = \sqrt{m^2-E}$, associated with eigenvalues of $D_\gamma$ that are larger than $m$.

Finally, to compute the limit of the solution $k_\gamma$ of~\eqref{equation_ev_A_gamma} for $\gamma \rightarrow \infty$, we note first that $\tanh(x) < 1$ for all $x>0$ and~\eqref{equation_ev_A_gamma} yield that $k_\gamma > -\frac{4 m \tau}{4+\tau^2}$ for all $\gamma > 0$. Thus, $\gamma k_\gamma \rightarrow \infty$, as $\gamma \rightarrow \infty$. By using now once more that $k_\gamma$ satisfies~\eqref{equation_ev_A_gamma}, we obtain
\begin{equation*}
  \lim_{\gamma \rightarrow \infty} k_\gamma = \lim_{\gamma \rightarrow \infty} \left(-\frac{4 m \tau}{4+\tau^2} \frac{1}{\tanh(k_\gamma \gamma)} \right) = -\frac{4 m \tau}{4+\tau^2}.
\end{equation*}
Hence,  we conclude for the smallest eigenvalue $E(\gamma)$ of $D_\gamma$ via the relation $k_\gamma = \sqrt{m^2 - E(\gamma)}$ that
\begin{equation*}
  \lim_{\gamma \rightarrow \infty} E(\gamma) = \lim_{\gamma \rightarrow \infty} (m^2 - k_\gamma^2) = m^2\left(\frac{\tau^2-4}{\tau^2+4}\right)^2,
\end{equation*}
which is exactly the claim in~\eqref{estimate_E_gamma}.
\end{appendix}

\subsection*{Acknowledgement}

DF and MH gratefully acknowledge financial support by the Austrian Science Fund (FWF): P33568-N. 
VL acknowledges the support by the grant No.~21-07129S of the Czech Science Foundation (GA\v{C}R).
This publication is based upon work from COST Action CA 18232 MAT-DYN-NET, supported by COST (European Cooperation in Science and Technology), www.cost.eu.


\end{document}